\documentclass[11pt]{article} 
\usepackage{amssymb}
\usepackage{graphicx}   
\usepackage{mathrsfs}
\usepackage{amsmath}
\usepackage{amsthm}
\usepackage{amsfonts}
\usepackage{latexsym} 
\usepackage[matrix,tips,graph,curve]{xy}

\makeatletter

\def\@maketitle{%
  \newpage
  \null
  \let \footnote \thanks
    {\normalfont\sffamily\bfseries\Large\noindent\@title \par}%
    \vskip 1em%
    {\normalfont\sffamily\large
        \noindent
        \@author
        \par}
  \par
  \vskip 4em}
\def\@seccntformat#1{\csname the#1\endcsname{.\ }}
\renewcommand\section{\@startsection {section}{1}{\z@}%
                                   {-3.0ex \@plus -1ex \@minus -.2ex}%
                                   {1.5ex \@plus.2ex}%
                                   {\normalfont\large\bfseries}}
\renewcommand\subsection{\@startsection{subsection}{2}{\z@}%
                                     {-2.75ex\@plus -1ex \@minus -.2ex}%
                                     {1.5ex \@plus .2ex}%
                                   {\normalfont\normalsize\bfseries}}
\makeatother
\makeatletter 
\@addtoreset{equation}{section}
\makeatother
\def\abstract{\topsep=0pt\partopsep=0pt\parsep=0pt\itemsep=0pt\relax
\trivlist\item[\hskip\labelsep
{\bfseries\abstractname}.]\if!\abstractname!\hskip-\labelsep\fi}
\if@twocolumn

\else
   
\fi

\theoremstyle{plain}
\newtheorem{theorem}[equation]{Theorem}
\newtheorem{corollary}[equation]{Corollary}
\newtheorem{lemma}[equation]{Lemma}

\theoremstyle{definition}
\newtheorem{definition}[equation]{Definition}

\newtheorem{remark}[equation]{Remark}

\renewcommand{\deg}{\mathrm{deg}}

\renewcommand\det{{\rm det\,}}

\def\d/{/\mspace{-6.0mu}/}

\begin{document}  

\title{Low-Order Geometric Actions with Fields a Metric and a Matter Field of Arbitrary Rank}

\author{Daniel Leeco Stern}

\maketitle
A Thesis Submitted to the Department of Mathematics for Honors

Duke University, Durham, NC 2014

\setcounter{section}{0}

\newpage

\emph{Acknowledgements:} To my mentor, Hugh Bray, a hearty thank you for suggesting this problem, and for the abundant guidance and support you've provided throughout this and many other research endeavors. Thanks also to Amir Aazami and Carla Cederbaum for many helpful conversations in the early stages of the project, and thanks to Dave Kraines and the PRUV program for their support in Summer 2012.

\hspace{6mm} Finally, thank you to my parents, my sister, and my wonderful friends, whose constant support and encouragement has time and again proved one of the most valuable items in my mathematical toolkit.

\newpage

\begin{abstract}
  We classify invariant Lagrangians of the form $L(g_{ij},g_{ij,k},g_{ij,kl},D_I,D_{I,j})$ depending at most quadratically on the variables $g_{ij,k},g_{ij,kl}$ and $D_I,D_{I,j}$, where $g$ is a Lorentz metric and $D$ is a tensor field of arbitrary rank on some smooth manifold. As a corollary, we prove a conjecture of Bray's regarding the classification of certain variational principles with variables a Lorentz metric and an affine connection. 
\end{abstract}

\newpage

\section{Introduction}\label{sec:introduction}

\hspace{6mm} In classical General Relativity, spacetime is described by a Lorentz manifold $(M^4,g)$ satisfying the Einstein equation
\begin{equation}\label{einsteineqn}
G=8\pi T,
\end{equation} where $G$ is the Einstein curvature tensor $G=Ric_g-\frac{1}{2}R_gg$ and $T$ is the stress-energy tensor of the matter fields present.\footnote{By $Ric_g,$ I of course mean the Ricci curvature $$Ric_{ij}=R^k_{ijk}=dx^k(\nabla_{\partial_k}\nabla_{\partial_j}\partial_i-\nabla_{\partial_j}\nabla_{\partial_k}\partial_i),$$ and by $R_g$ its contraction with $g$, i.e. the scalar curvature.} While the tensor formula \eqref{einsteineqn} provides one realization of the qualitative principle at the center of GR--that matter curves spacetime, it's not clear, at first glance, why this is a more natural choice than, say, $Ric=8\pi T$, or other formulas relating curvature to matter density. The key reason $G$ gives the most natural representative of spacetime curvature for these purposes lies in the standard symmetries of the curvature operator: by the second Bianchi identity, the divergence of $G$ automatically vanishes in every spacetime--a statement which, coupled with \eqref{einsteineqn}, we can view as a conservation principle.

\hspace{6mm} While the fact that $div (G)=0$ is an immediate consequence of the second Bianchi identity, it also follows from a fundamental observation made by Hilbert in the early days of GR: the Einstein tensor $G$ is the Euler-Lagrange tensor associated with the action 
\begin{equation}\label{ehaction}
E(g)=\int_MR_gdVol_g;
\end{equation}
that is, $G$ is the unique tensor field satisfying
\begin{equation}\label{eldef}
\frac{d}{dt}\int_KR_{g_t}dVol_{g_t}|_{t=0}=-\int_K\langle G,h\rangle_gdVol_g
\end{equation}
for all varations $g_t$ of $g$ supported in a compact set $K\subset M$, where $h=\frac{\partial g_t}{\partial t}|_{t=0}$ \cite{LR}.\footnote{Strictly speaking, the total integral $E(g)$ as written is generally not defined in this context, since our spacetimes are usually taken to be noncompact. Instead, we can think of $E(g)$ as representing the family of (well-defined) functionals given by integrating the scalar curvature over compact subsets of spacetime. Similar abuses of notation will occur frequently throughout this discussion.} In general, for any action of the form
$$F(g)=\int_Mf_gdVol_g$$
where $f_g$ has the coordinate expression $f_g=|\det (g_{ij})|^{-1/2}L(g_{ij},g_{ij,k},g_{ij,kl})$, it's not difficult to see that the Euler-Lagrange tensor (given in each chart $(x^i)$ by $E^{ij}=|\det(g_{ij})|^{-1/2}(-\frac{\partial L}{\partial g_{ij}}+\frac{\partial}{\partial x^k}\frac{\partial L}{\partial g_{ij,k}}-\frac{\partial^2}{\partial x^k\partial x^l}\frac{\partial L}{\partial g_{ij,kl}})$) is divergence-free \cite{LR}. One key advantage of Hilbert's formulation of GR is that it provides a natural way of coupling a matter field's equations of motion with the Einstein equation governing the geometry of the spacetime: if a matter field is described by some ``potential" $(r,s)$-tensor field $A \in \mathscr{T}^r_s(M),$ we consider actions of the form
$$F(g,A)=\int_M(R_g+f_{g,A})dVol_g,$$
and call a pair $(g,A)$ physical if and only if it is a critical point of $F$ with respect to compactly supported variations of both $g$ and $A$. Varying in the $g$ direction yields the Einstein equation $G=8\pi T$ (where we define the stress energy tensor $T$ in a manner similar to \eqref{eldef}--up to a factor of $-8\pi$--by replacing $R_{g_t}$ with $f_{g_t,A}$), while varying in the $A$ direction yields the equations of motion of the matter field--e.g., Maxwell's equations, when $A$ is the one-form describing electromagnetic potential \cite{LR}.

\hspace{6mm} Given the success of this Lagrangian formulation of GR, it is natural to ask how the theory changes when we base our theory not on the Einstein-Hilbert action \eqref{ehaction}, but on structurally similar functionals of the metric. If we define ``structurally similar" actions to be those of the form $\int_ML(g_{ij},g_{ij,k},g_{ij,kl})dVol_g,$ where $L(g_{ij},g_{ij,k},g_{ij,kl})$ is a coordinate-invariant function depending (affine-)linearly on the second derivatives of the metric, the answer is largely understood. It was shown around a century ago by Cartan and Weyl that the space of functions $L(g_{ij},g_{ij,k},g_{ij,kl})$ satisfying these conditions is spanned by the scalar curvature $R_g$ and the constant function $1$ \cite{Weyl}. Thus, all functionals structurally similar to \eqref{ehaction} in the above sense have the form
\begin{equation}\label{ehwithconstant}
\int_M(aR_g+b)dVol_g,\text{ \hspace{4mm}}(a,b \in \mathbb{R})
\end{equation}
so that, up to scaling (except in the degenerate case $a=0$), the associated Euler-Lagrange tensor has the form
\begin{equation}\label{ehwithconstant2}
G+\Lambda g
\end{equation}
for some constant $\Lambda$. In recent decades, astrophysicists have observed that a nontrivial cosmological constant--as the parameter $\Lambda$ in \eqref{ehwithconstant2} is now known--can account for the accelerating expansion of the universe, and modifying the classical Einstein-Hilbert action by adding a small constant term has become one of the most popular explanations for the phenomenon known as ``dark energy" \cite{BS}.  

\hspace{6mm} While the classification of geometric objects subject to certain structural constraints is a compelling subject in its own right, the example above illustrates how these kinds of classification problems can be of direct physical import: by exploring all those theories which are, in some sense, ``close to" a standard one, we may discover a modification that resolves a problem in the original theory.

\hspace{6mm} In \cite{Br}, Bray constructs an intriguing model for the gravitational effects of dark matter by considering functionals $F(g,\nabla)$ of a Lorentz metric $g$ and an affine connection $\nabla,$ given in coordinates by a Lagrangian of the form
\begin{equation}\label{hughlag}
L(g_{ij},g_{ij,k},\Gamma_{ijk},\Gamma_{ijk,l})=Quad_{(g_{ij})}(g_{ij,k},\Gamma_{ijk},\Gamma_{ijk,l})
\end{equation}
--that is, degree-2 polynomials in the variables $(g_{ij,k},\Gamma_{ijk},\Gamma_{ijk,l})$ with coefficients in the smooth functions of the metric components $(g_{ij})$.
Specifically, he shows that the subset of these which determine a variational principle equivalent to one given by an action of the form
\begin{equation}\label{hughlag2}
\int_M(aR_g+b+c|d\gamma|_g^2+Q_g(D))dVol_g
\end{equation}
predict a number of cosmological phenomena often attributed to dark matter--most notably, the spiral patterns found in many disk galaxies \cite{Br}.\footnote{Of course, the Lagrangians of the form \eqref{hughlag} do not, in fact, include the scalar curvature $R_g,$ as the scalar curvature depends nontrivially on second derivatives of the metric; however, $R_g$ differs from a Lagrangian of the given form in each coordinate chart only by a divergence term, and therefore (by the divergence theorem) determines an equivalent variational principle, with the caveat that we only consider variations supported in coordinate neighborhoods.} (Here, following Bray's notation, $D$ is the difference tensor 
$$D_{ijk}:=\Gamma_{ijk}-\frac{1}{2}(g_{ik,j}+g_{jk,i}-g_{ij,k}),$$ $\gamma$ is the fully-antisymmetric part of $D$, and $Q_g(D)$ is a coordinate-invariant function given in coordinates by a quadratic polynomial in $(D_{ijk})$ with coefficients in the smooth functions of $(g_{ij})$.) Bray then conjectures that all variational principles determined by Lagrangians of the form \eqref{hughlag} are in fact equivalent to one of those given by the actions \eqref{hughlag2}, so that all actions with a polynomial structure similar to that of \eqref{hughlag2} will yield the same dark matter model.

\hspace{6mm} By replacing the assumption that $L$ take the form \eqref{hughlag} with the requirement that $L$ take the more general form\footnote{The reason for allowing the more general structure (\ref{danhyp}) is that it allows us to pass back and forth between the connection and difference tensor terms without altering the structure of the Lagrangian.}
\begin{equation}\label{danhyp}
L(g_{ij},g_{ij,k},g_{ij,kl},\Gamma_{ijk},\Gamma_{ijk,l})=Quad_{(g_{ij})}(g_{ij,k},g_{ij,kl},\Gamma_{ijk},\Gamma_{ijk,l})
\end{equation}
and satisfy a coordinate invariance condition, we answer this question in the affirmative, and establish that, more generally,

\begin{theorem}\emph{(Main Theorem--Paraphrase of Theorem \ref{myscalarthm})}\label{mainthmpara} Any geometric variational principle with fields a metric and a tensor field of type $(0,r)$ given by a coordinate-invariant Lagrangian of the form 
\begin{equation}\label{genlagrange}
L(g_{ij},g_{ij,k},g_{ij,kl},D_I,D_{I,j})=Quad_{(g_{ij})}(g_{ij,k},g_{ij,kl},D_I,D_{I,j})
\end{equation}
is equivalent to one given by a Lagrangian of the form
$$a+bR_g+c|d\gamma|_g^2+T_g(D)+Q_g(D),$$
where $Q_g(D)=\mu^{IJ}(g_{ab})D_ID_J$ and $T_g(D)=\eta^I(g_{ab})D_I$ are coordinate-invariant quadratic and linear functions, respectively, of $(D_I)$, with coefficients in the smooth functions of the metric components $(g_{ij}).$
\end{theorem}

In particular, we deduce that, for all varational principles of the form (\ref{genlagrange}), the only part of $D$ whose dynamics are controlled by the resulting equations of motion is the fully antisymmetric part $\gamma$.

\medskip

\section{Review of Classical Results}\label{sec:classic}

\hspace{6mm} In order to introduce many of the techniques that we'll employ in the proof of the main theorem, we review here some classical results concerning the classification of certain families $\{f_g\}$ of smooth functions on Lorentz manifolds $(M^n,g)$, where $f_g$ is described in each coordinate chart by a function $L(g_{ij},g_{ij,k},g_{ij,kl})$ of the metric and its derivatives. Our proofs of these theorems follow the sketches given by Weyl in \cite{Weyl}, but fill in a number of details often left out of the literature.

\hspace{6mm} As a matter of convenience, we first briefly recall some standard transformation formulas for the coordinate representations of semi-Riemannian metrics. Let $(x^i)$ and $(\tilde{x}^i)$ be overlapping coordinate charts on a manifold $M$, and let $g$ be a semi-Riemannian metric on $M$, described in each chart by the components $g_{ij}$ and $\tilde{g}_{ij},$ respectively. Then, on the overlap of these charts, we have:
\begin{equation}\label{gijs}
\tilde{g}_{ij}=\frac{\partial x^a}{\partial \tilde{x}^i}\frac{\partial x^b}{\partial \tilde{x}^j}g_{ab},
\end{equation}
\begin{equation}\label{gijks}
\tilde{g}_{ij,k}=\frac{\partial x^a}{\partial\tilde{x}^i}\frac{\partial x^b}{\partial \tilde{x}^j}\frac{\partial x^c}{\partial \tilde{x}^k}g_{ab,c}+\left(\frac{\partial^2 x^a}{\partial \tilde{x}^i\partial\tilde{x}^k}\frac{\partial x^b}{\partial \tilde{x}^j}+\frac{\partial x^a}{\partial \tilde{x}^i}\frac{\partial^2x^b}{\partial \tilde{x}^j\partial\tilde{x}^k}\right)g_{ab},
\end{equation}
and
\begin{equation}\label{gijkls}
\begin{split}
\tilde{g}_{ij,kl}= & \frac{\partial x^a}{\partial\tilde{x}^i}\frac{\partial x^b}{\partial \tilde{x}^j}\frac{\partial x^c}{\partial \tilde{x}^k}\frac{\partial x^d}{\partial\tilde{x}^l}g_{ab,cd}\\
&+\left(\frac{\partial^2 x^a}{\partial \tilde{x}^i\partial\tilde{x}^l}\frac{\partial x^b}{\partial \tilde{x}^j}\frac{\partial x^c}{\partial \tilde{x}^k}+\frac{\partial x^a}{\partial \tilde{x}^i}\frac{\partial^2 x^b}{\partial \tilde{x}^j\partial\tilde{x}^l}\frac{\partial x^c}{\partial\tilde{x}^k}+\frac{\partial x^a}{\partial \tilde{x}^i}\frac{\partial x^b}{\partial \tilde{x}^j}\frac{\partial^2x^c}{\partial \tilde{x}^k\partial \tilde{x}^l}\right)g_{ab,c}\\
&+\frac{\partial x^c}{\partial \tilde{x}^l}\left(\frac{\partial^2 x^a}{\partial \tilde{x}^i\partial\tilde{x}^k}\frac{\partial x^b}{\partial \tilde{x}^j}+\frac{\partial x^a}{\partial \tilde{x}^i}\frac{\partial^2x^b}{\partial \tilde{x}^j\partial\tilde{x}^k}\right)g_{ab,c}\\
&+\left(\frac{\partial^2 x^a}{\partial \tilde{x}^i\partial\tilde{x}^k}\frac{\partial^2x^b}{\partial \tilde{x}^j\partial \tilde{x}^l}+\frac{\partial^2 x^a}{\partial \tilde{x}^i\partial\tilde{x}^l}\frac{\partial^2 x^b}{\partial \tilde{x}^j\partial \tilde{x}^k}\right)g_{ab}\\
&+\left(\frac{\partial^3 x^a}{\partial \tilde{x}^i\partial\tilde{x}^k\partial\tilde{x}^l}\frac{\partial x^b}{\partial \tilde{x}^j}+\frac{\partial x^a}{\partial \tilde{x}^i}\frac{\partial^3x^b}{\partial \tilde{x}^j\partial \tilde{x}^k\partial\tilde{x}^l}\right)g_{ab}.
\end{split}
\end{equation}

\hspace{6mm} Before we can state the results in question, we need to specify the appropriate domain for these functions $L$. Those readers most familiar with the calculus of variations will recognize that the natural domain for the Lagrangians of interest is the 2-jet bundle $J^2(E)$ associated with the bundle $Sym^2(T^*M)\supset E \to M$ of nondegenerate symmetric $(0,2)$-tensors on an underlying manifold $M$, so that the integrand in the associated action is obtained from $L$ simply by composition with sections of $J^2(E)$. However, we avoid this treatment here for two reasons:

a. Taking the domain of $L$ to be $J^2(E)$ presupposes the coordinate invariance of $L$. For our purposes, it will be more useful to introduce coordinate invariance as an explicit algebraic condition, rather than a property encoded in the domain.

b. We wish to keep the discussion sufficiently elementary that any reader familiar with the most basic definitions of semi-Riemannian geometry will be able to follow. Though jet bundles play a fundamental role in the modern study of the calculus of variations,\footnote{And for a good introduction to jet bundles, we refer the interested reader to \cite{S}.} introducing them here would be a digression that would likely serve more to confuse than to clarify.

Furthermore, since the Lagrangians of interest do not depend on the base space with respect to local trivializations of $J^2(E\to M)$, what follows is in essence a discussion of how best to coordinatize the space $J^2(E \to \mathbb{R}^n)_0$ of $2$-jets at the origin of $\mathbb{R}^n$.

\hspace{6mm} In an obvious way, we can view $(g_{ij}(p),g_{ij,k}(p),g_{ij,kl}(p))$ as an element of $V_n:=(\mathbb{R}^n)^{\otimes 2}\times(\mathbb{R}^n)^{\otimes 3}\times(\mathbb{R}^n)^{\otimes 4}\cong\mathbb{R}^{n^2}\times\mathbb{R}^{n^3}\times\mathbb{R}^{n^4}$ (making the identification $(w_{ij},w_{ijk},w_{ijkl})=(w_{ij}e_i\otimes e_j, w_{ijk}e_i\otimes e_j\otimes e_k,w_{ijkl}e_i\otimes e_j\otimes e_k\otimes e_l)$, where all indices run from $0$ to $n-1$ and the standard summation convention is in effect). Defining
\begin{equation}\label{wdef}
W_n:=\{(w_{ij},w_{ijk},w_{ijkl}) \in V_n \mid \det(w_{ij}) \neq 0\},
\end{equation}
we see that $(g_{ij}(p),g_{ij,k}(p),g_{ij,kl}(p)) \in W_n$ as well, by the nondegeneracy of the metric tensor. It's important that we restrict our attention to functions on $W_n$ rather than $V_n$, as most of the functions $L(w_{ij},w_{ijk},w_{ijkl})$ of interest--in particular, the function $L: W_n \to \mathbb{R}$ such that $L(g_{ij},g_{ij,k},g_{ij,kl})=R_g$--employ the inverse of the matrix $(w_{ij})$ in their construction. We'll use $W_n$ as the domain for our functions $L$, largely for reasons of notational convenience. (Because of the symmetries and signature of the terms $(g_{ij})$, $(g_{ij,k})$, and $(g_{ij,kl})$, we could of course restrict our attention to smaller domains, but this would be an unnecessary complication, and would have no effect on the following results.) 

\hspace{6mm} With these definitions in place, we can now state the classic result of Cartan and Weyl regarding the uniqueness of scalar curvature:

\begin{theorem}\label{thm01}\emph{(\cite{Weyl})}
 Let $L:W_n\to\mathbb{R}$ be a smooth function of the form   
$$L(w_{ij},w_{ijk},w_{ijkl})=\alpha^{ijkl}(w_{ab},w_{abc})w_{ijkl}+\beta(w_{ab},w_{abc})$$ 
such that, for  every Lorentz n-manifold $(M^n, g)$, there exists a (globally-defined) function $f_g \in C^{\infty}(M)$ satisfying \begin{equation}\label{thm01invariance}
L(g_{ij},g_{ij,k},g_{ij,kl})=f_g
\end{equation} in every coordinate chart on $M$. Then there are constants $a,b \in \mathbb{R}$ such that $f_g=aR_g+b$ for every Lorentz manifold $(M^n,g)$.
\end{theorem}

\begin{remark}\label{rk1} 
Since we're only interested in the restriction of $L$ to the components of symmetric tensor fields and their derivatives, we can assume without loss of generality (by considering instead $L\circ \varphi$, where $\varphi(w_{ij},w_{ijk},w_{ijkl})=\frac{1}{2}(w_{ij}+w_{ji},w_{ijk}+w_{jik},\frac{1}{2}(w_{ijkl}+w_{jikl}+w_{ijlk}+w_{jilk}))$) that $L$ satisfies the symmetries 

\begin{equation}\label{symmetries1}
\frac{\partial L}{\partial w_{ij}}=\frac{\partial L}{\partial w_{ji}},\text{ }
\frac{\partial L}{\partial w_{ijk}}=\frac{\partial L}{\partial w_{jik}},\text{ and }
\frac{\partial L}{\partial w_{ijkl}}=\frac{\partial L}{\partial w_{jikl}}=\frac{\partial L}{\partial w_{ijlk}}.
\end{equation}
\end{remark}

\hspace{6mm} An important first step in the proof will be to establish coordinate invariance (in the tensorial sense) of the coefficients $\alpha^{ijkl}=\frac{\partial L}{\partial w_{ijkl}}$; though the proof of this lemma is not particularly subtle, we carry it out in full detail below, as we will employ analogous statements without proof in later sections. 

\begin{lemma}\label{weyllemma} Let $L:W_n \to \mathbb{R}$ be a smooth function satisfying the invariance hypothesis \eqref{thm01invariance} of Theorem \ref{thm01} and the symmetries \eqref{symmetries1}. Then, for every Lorentz manifold $(M^n, g)$, the relation
\begin{equation}\label{scalardiffinvar1}
\frac{\partial L}{\partial w_{ijkl}}(g_{ij},g_{ij,k},g_{ij,kl})=\frac{\partial x^i}{\partial \tilde{x}^a}\frac{\partial x^j}{\partial \tilde{x}^b}\frac{\partial x^k}{\partial \tilde{x}^c}\frac{\partial x^l}{\partial \tilde{x}^d}\frac{\partial L}{\partial w_{abcd}}(\tilde{g}_{ij},\tilde{g}_{ij,k},\tilde{g}_{ij,kl}).
\end{equation}
holds on the overlap of any two coordinate charts $(x^i)$ and $(\tilde{x}^i)$--that is, the derivatives $\frac{\partial L}{\partial w_{ijkl}}(g_{ij},g_{ij,k},g_{ij,kl})$ (evaluated at the metric components and their derivatives) form the components of a $(4,0)$-tensor field.
\end{lemma}

\begin{proof} Let $M^n$ be a manifold admitting a Lorentz metric. Fix an arbitrary point $p \in M$ and a coordinate chart $\xi=(x^i)$ defined on a neighborhood of $p$. For every coordinate chart $\tilde{\xi}=(\tilde{x}^i)$ defined on a neighborhood of $p$, define a map $\Phi_{\tilde{\xi}}:W_n\to W_n$ by 
$$\Phi_{\tilde{\xi}}(w_{ij},w_{ijk},w_{ijkl})=(\tilde{w}_{ij},\tilde{w}_{ijk},\tilde{w}_{ijkl}),$$
where $(\tilde{w}_{ij},\tilde{w}_{ij,k},\tilde{w}_{ij,kl})$ is given by equations \eqref{gijs}, \eqref{gijks}, and \eqref{gijkls} (replacing $g$ with $w$ and evaluating the derivatives of $\xi\circ\tilde{\xi}^{-1}$ at $p$), so that, if $T$ is any (0,2)-tensor field on $M$, we have 
$$\Phi_{\tilde{\xi}}(T_{ij}(p),T_{ij,k}(p),T_{ij,kl}(p))=(\tilde{T}_{ij}(p),\tilde{T}_{ij,k}(p),\tilde{T}_{ij,kl}(p)).$$
Then, for any Lorentz metric $g$ on $M$, it follows from the hypotheses of Theorem \ref{thm01} that
\begin{equation}\label{scalarinvar1}
L \circ \Phi_{\tilde{\xi}}(g_{ij}(p),g_{ij,k}(p),g_{ij,kl}(p))=f_g(p)=L(g_{ij}(p),g_{ij,k}(p),g_{ij,kl}(p)).
\end{equation}

Fix some Lorentz metric $g$ on $M$. For any $(c_{ijkl}) \in (\mathbb{R}^n)^{\otimes 4}$ satsifying the symmetries $c_{ijkl}=c_{jikl}=c_{ijlk}$, it's easy to construct another Lorentz metric $h$ on $M$ such that $$(h_{ij}(p),h_{ij,k}(p),h_{ij,kl}(p))=(g_{ij}(p),g_{ij,k}(p),c_{ijkl}),$$ and thus, applying \eqref{scalarinvar1} to the metric $h$, it follows that
$$L\circ \Phi_{\tilde{\xi}}(g_{ij}(p),g_{ij,k}(p),c_{ijkl})=L(g_{ij}(p),g_{ij,k}(p), c_{ijkl})$$
for all such $(c_{ijkl})$. Consequently, setting $e_{ijkl}:=e_i\otimes e_j\otimes e_k\otimes e_l \in (\mathbb{R}^n)^{\otimes 4}$ and defining $v_{ijkl}:=\frac{1}{4}(e_{ijkl}+e_{jikl}+e_{ijlk}+e_{jilk})$, we have
\begin{equation}\label{secordinv}
L \circ \Phi_{\tilde{\xi}}(g_{ij}(p),g_{ij,k}(p),g_{ij,kl}(p)+tv_{ijkl})=L(g_{ij}(p),g_{ij,k}(p),g_{ij,kl}(p)+tv_{ijkl})
\end{equation}
for all $t \in \mathbb{R}$. Differentiating (\ref{secordinv}) in $t$, we obtain (setting $x=(g_{ij}(p),g_{ij,k}(p),g_{ij,kl}(p))$ for convenience)
\begin{eqnarray*}
D_{v_{ijkl}}L(x)&=&D_{v_{ijkl}}(L\circ \Phi_{\tilde{\xi}})(x)\\
&=&\frac{\partial L}{\partial w_{ab}}(\Phi_{\tilde{\xi}}(x))D_{v_{ijkl}}\tilde{w}_{ab}(x)+\frac{\partial L}{\partial w_{abc}}(\Phi_{\tilde{\xi}}(x))D_{v_{ijkl}}\tilde{w}_{abc}(x)\\
&&+\frac{\partial L}{\partial w_{abcd}}(\Phi_{\tilde{\xi}}(x))D_{v_{ijkl}}\tilde{w}_{abcd}(x)\\
&=&\frac{\partial L}{\partial w_{abcd}}(\Phi_{\tilde{\xi}}(x))D_{v_{ijkl}}\tilde{w}_{abcd}(x)\\
&=&\frac{1}{4}\left(\frac{\partial L}{\partial w_{abcd}}(\Phi_{\tilde{\xi}}(x))\frac{\partial \tilde{w}_{abcd}}{\partial w_{ijkl}}(x)+\cdots+\frac{\partial L}{\partial w_{abcd}}(\Phi_{\tilde{\xi}}(x))\frac{\partial \tilde{w}_{abcd}}{\partial w_{jilk}}(x)\right)\\
\text{(by \eqref{symmetries1})}&=&\frac{1}{4}\left(\frac{\partial L}{\partial w_{abcd}}(\Phi_{\tilde{\xi}}(x))\frac{\partial \tilde{w}_{abcd}}{\partial w_{ijkl}}(x)+\cdots+\frac{\partial L}{\partial w_{badc}}(\Phi_{\tilde{\xi}}(x))\frac{\partial \tilde{w}_{abcd}}{\partial w_{jilk}}(x)\right)\\
&=&\frac{1}{4}\left(\frac{\partial L}{\partial w_{abcd}}(\Phi_{\tilde{\xi}}(x))\frac{\partial \tilde{w}_{abcd}}{\partial w_{ijkl}}(x)+\cdots+\frac{\partial L}{\partial w_{abcd}}(\Phi_{\tilde{\xi}}(x))\frac{\partial \tilde{w}_{badc}}{\partial w_{jilk}}(x)\right)\\
\text{(by \eqref{gijkls})}&=&\frac{\partial L}{\partial w_{abcd}}(\Phi_{\tilde{\xi}}(x))\frac{\partial x^i}{\partial \tilde{x}^a}(p)\frac{\partial x^j}{\partial \tilde{x}^b}(p)\frac{\partial x^k}{\partial \tilde{x}^c}(p)\frac{\partial x^l}{\partial \tilde{x}^d}(p).
\end{eqnarray*}
Finally, since $D_{v_{ijkl}}=\frac{1}{4}\left(\frac{\partial}{\partial w_{ijkl}}+\frac{\partial}{\partial w_{jikl}}+\frac{\partial}{\partial w_{ijlk}}+\frac{\partial}{\partial w_{jilk}}\right)$, the relations above, together with the symmetries \eqref{symmetries1} give us
\begin{eqnarray*}
\frac{\partial L}{\partial w_{ijkl}}(g_{ij}(p),g_{ij,k}(p),g_{ij,kl}(p))&=&D_{v_{ijkl}}(g_{ij}(p),g_{ij,k}(p),g_{ij,kl}(p))\\
&=&\frac{\partial x^i}{\partial \tilde{x}^a}(p)\frac{\partial x^j}{\partial \tilde{x}^b}(p)\frac{\partial x^k}{\partial \tilde{x}^c}(p)\frac{\partial x^l}{\partial \tilde{x}^d}(p)\frac{\partial L}{\partial w_{abcd}}(\tilde{g}_{ij}(p),\tilde{g}_{ij,k}(p),\tilde{g}_{ij,kl}(p)),
\end{eqnarray*}
as desired.
\end{proof}

\begin{remark}\label{derivtensoriality1} By a slight extension of the above argument, we can show that, for any tensor-valued function $F:W_n\to (\mathbb{R}^n)^{\otimes r}\otimes(\mathbb{R}^n)^{\otimes s}$ satisfying the invariance hypothesis
\begin{equation}
F^I_J(g_{ij},g_{ij,k},g_{ij,kl})=(T_g)^I_J\text{ (}T_g\text{ a fixed }(r,s)\text{-tensor field) in all coordinates on }M,
\end{equation}
the derivatives $\frac{\partial F}{\partial w_{ijkl}}(g_{ij},g_{ij,k},g_{ij,kl})$ give the components of a type $(r+4,s)$-tensor field for every Lorentz manifold $(M^n,g)$. In particular, it follows by induction that, for a function $L$ satisfying the hypotheses of Lemma \ref{weyllemma}, the derivatives $\frac{\partial^pL}{\partial w_{i_1j_1k_1l_1}\cdots \partial w_{i_pj_pk_pl_p}}(g_{ij},g_{ij,k},g_{ij,kl})$ form the components of a type $(4p,0)$-tensor field. For the functions of interest in Theorem \ref{thm01}, this is of course irrelevant, since all such higher derivatives vanish, but this generalization (particularly the case $p=2$) will be useful in establishing other results in the sequel.
\end{remark}

\hspace{6mm} Before proving Theorem \ref{thm01}, let's fix notation and recall some standard constructions from semi-Riemannian geometry. (See, e.g., \cite{ONe}.)

\begin{definition} \label{def:norm_coords} Let $p$ be a point in a semi-Riemannian manifold $(M^n,g)$. A coordinate system $\xi: U \to \mathbb{R}^n$ defined on a neighborhood $U \subset M$ of $p$ is said to be \emph{normal} at $p$ if $\xi(p)=0$ and $g_{ij,k}(p)=0$ in these coordinates (equivalently, if the Christoffel symbols $\Gamma_{ij}^k$ of the Levi-Civita connection vanish at $p$). Moreover, if $(M^n,g)$ is a Lorentz manifold and $\xi$ is a normal coordinate system at $p$ for which $g_{ij}(p)=\eta_{ij}$ (where $\eta_{ij}$ is the $ij$th component of the diagonal $n\times n$ matrix $diag(-1,1,1,\ldots,1)$), then we'll call $\xi$ a \emph{Lorentz} normal coordinate system at $p$. 

The exponential map can always be used to construct normal coordinates at any point on a semi-Riemannian manifold, and the existence of Lorentz normal coordinates at each point of a Lorentz manifold follows by applying an appropriate linear coordinate transformation (or simply starting from an orthonormal basis on $T_pM$). These coordinates will play an essential role in the following proof, which closely follows Weyl's original argument in \cite{Weyl}. 
\end{definition}

\begin{proof}[Proof of Theorem \ref{thm01}] Without loss of generality (see Remark \ref{rk1}), assume that $L$ satisfies the symmetries \eqref{symmetries1}. Define constants $b$ and $a^{ijkl} \in \mathbb{R}$ by 
$$b:=\beta(\eta_{ab},0)\text{ and }a^{ijkl}:=\alpha^{ijkl}(\eta_{ab},0),\text{ (with }(\eta_{ab}) \in \mathbb{R}^{n^2}\text{ defined as before, and }0 \in \mathbb{R}^{n^3})$$
so that for any point $p$ in a Lorentz manifold $(M^n,g)$, in any Lorentz normal coordinate system at $p$, we have
$$f_g(p)=L(g_{ij}(p),g_{ij,k}(p),g_{ij,kl}(p))=a^{ijkl}g_{ij,kl}(p)+b.$$
Now we just need to show that the constants $a^{ijkl}$ automatically satisfy $a^{ijkl}g_{ij,kl}(p)=aR_g(p)$ in all normal coordinate systems, for some $a \in \mathbb{R}$.

\hspace{6mm} For completeness, let's begin by recalling the form of the scalar curvature $R_g(p)$ of a Lorentz manifold $(M^n,g)$ in Lorentz normal coordinates at $p \in M$. Letting $\nabla$ denote the Levi-Civita connection induced by the metric $g$, the Riemann curvature is given in Lorentz normal coordinates at a point $p$ by 
\begin{eqnarray*}
R^i_{jkl}(p)&=&dx^i(\nabla_{\partial_l}\nabla_{\partial_k}\partial_j-\nabla_{\partial_k}\nabla_{\partial_l}\partial_j)(p)\\
&=&dx^i(\nabla_{\partial_l}(\Gamma_{kj}^r\partial_r)-\nabla_{\partial_k}(\Gamma_{jl}^r\partial_r))(p)\\
&=&dx^i(\frac{\partial \Gamma_{kj}^r}{\partial x^l}\partial_r+\Gamma_{kj}^r\Gamma_{rl}^s\partial_s-\frac{\partial \Gamma_{jl}^r}{\partial x^k}\partial_r-\Gamma_{jl}^r\Gamma_{rk}^s\partial_s)(p)\\
&=&\frac{\partial \Gamma^i_{jk}}{\partial x^l}(p)-\frac{\partial \Gamma_{jl}^i}{\partial x^k}(p)\\
&=&\frac{\partial}{\partial x^l}(\frac{1}{2}g^{ia}(g_{ja,k}+g_{ka,j}-g_{jk,a}))(p)-\frac{\partial}{\partial x^k}(\frac{1}{2}g^{ia}(g_{ja,l}+g_{la,j}-g_{jl,a}))(p)\\
&=& \frac{1}{2}g^{ia}(p)(g_{ja,kl}+g_{ka,jl}-g_{jk,al}-g_{ja,lk}-g_{la,jk}+g_{jl,ak})(p),
\end{eqnarray*}
since $\frac{\partial g^{ij}}{\partial x^k}(p)=\frac{\partial g^{ij}}{\partial g_{ab}}(p)g_{ab,k}(p)=0$ in normal coordinates at $p$.
Hence,
\begin{eqnarray*}
R(p)&=&g^{jk}(p)R^i_{jki}(p)\\
&=&\frac{1}{2}g^{jk}(p)g^{ia}(p)(g_{ja,ki}+g_{ka,ji}-g_{jk,ai}-g_{ja,ik}-g_{ia,jk}+g_{ji,ak})(p)\\
&=&\frac{1}{2}g^{jk}(p)g^{ia}(p)(2g_{ij,ka}-2g_{ia,jk}),
\end{eqnarray*}
and since $g^{ij}(p)=\eta_{ij}$ in normal coordinates, we have
\begin{equation}\label{scalarcomp}
R(p)=\eta_{jk}\eta_{il}(g_{ij,lk}-g_{il,jk})=\epsilon_i\epsilon_j(g_{ij,ij}-g_{ii,jj}),
\end{equation}
where $(\epsilon_0,\epsilon_1,\ldots,\epsilon_{n-1})=(-1,1,\ldots,1)$. (In general, if the metric $g$ has signature $(\epsilon_1,\ldots,\epsilon_n)=(-1,\ldots,-1,1,\ldots,1)$, we can choose normal coordinates at $p$ in which $g_{ij}(p)=\epsilon_i\delta_{ij}$, and the above computation still holds.)

\hspace{6mm} Now, fix an arbitrary Lorentz manifold $(M^n,g)$, and let $(x^i)$ be a Lorentz normal coordinate system at a point $p \in M$. Given $(c^i_{jkl})=(c_{ijkl}) \in (\mathbb{R}^n)^{\otimes 4}$ satisfying $c^i_{jkl}=c^i_{kjl}=c^i_{jlk},$ the inverse function theorem guarantees the existence of a smooth coordinate system $(\tilde{x}^i)$ with $\tilde{x}^i(p)=0$, defined implicitly on a neighborhood of $p$ by
$$x^i=\tilde{x}^i+\frac{1}{6}\eta_{ia}c^a_{jkl}\tilde{x}^j\tilde{x}^k\tilde{x}^l.$$
Since $\tilde{x}^i(p)=x^i(p)=0$, we easily compute
\begin{equation}\label{abovecomp}\frac{\partial x^i}{\partial \tilde{x}^j}(p)=\delta^i_j,\text{ }\frac{\partial^2 x^i}{\partial \tilde{x}^j\partial \tilde{x}^k}(p)=0\text{, and }\frac{\partial^3 x^i}{\partial \tilde{x}^j\partial\tilde{x}^k\partial\tilde{x}^l}(p)=\eta_{ia}c^a_{jkl}.
\end{equation}
Letting $\tilde{g}_{ij}$ denote the metric components with respect to the coordinates $(\tilde{x}^i)$ and applying (\ref{abovecomp}) to the transformation formulas \eqref{gijs}, \eqref{gijks}, and \eqref{gijkls}, we obtain
\begin{equation}\label{weyltrans1}
\tilde{g}_{ij}(p)=g_{ij}(p)=\eta_{ij},\text{ }\tilde{g}_{ij,k}(p)=g_{ij,k}(p)=0,\text{ and }
\end{equation}
\begin{equation}\label{weyltrans2}
\tilde{g}_{ij,kl}=g_{ij,kl}(p)+(\eta_{ac}c^c_{ikl}\delta_j^b+\eta_{bc}\delta_i^ac^c_{jkl})\eta_{ab}=g_{ij,kl}(p)+c^j_{ikl}+c^i_{jkl}.
\end{equation}
From the relations \eqref{weyltrans1}, we know that $(\tilde{x}^i)$ is a Lorentz normal coordinate system at $p$, and it therefore follows from \eqref{weyltrans2} that
$$a^{ijkl}g_{ij,kl}(p)+b=f_g(p)=a^{ijkl}\tilde{g}_{ij,kl}(p)+b=a^{ijkl}(g_{ij,kl}(p)+c^j_{ikl}+c^i_{jkl})+b.$$
Thus, the coefficients $a^{ijkl}$ must satisfy
$$a^{ijkl}(c^j_{ikl}+c^i_{jkl})=0,$$
and since $a^{ijkl}=\frac{\partial L}{\partial w_{ijkl}}(\eta_{ab},0)$, it follows from \eqref{symmetries1} that 
\begin{equation}\label{symm1again}
a^{ijkl}=a^{jikl}=a^{ijlk},
\end{equation}
and consequently,
\begin{equation}\label{weylsym1}
0=a^{ijkl}c^j_{ikl}+a^{ijkl}c^i_{jkl}=a^{jikl}c^j_{ikl}+a^{ijkl}c^i_{jkl}=2a^{ijkl}c^i_{jkl}.
\end{equation}
 Given $0 \leq i,j,k,l \leq n-1$, define $(c^r_{stu}) \in \mathbb{R}^{\otimes 4}$ by 
$$c^i_{jkl}=c^i_{kjl}=c^i_{jlk}=c^i_{lkj}=c^i_{klj}=c^i_{ljk}=1\text{ and }c^r_{stu}=0\text{ otherwise.}$$
Applying \eqref{weylsym1} to this choice of $(c^r_{stu})$ then yields
$$a^{ijkl}+a^{ikjl}+a^{ijlk}+a^{ilkj}+a^{iljk}+a^{iklj}=0,$$
and therefore, by \eqref{symm1again},
\begin{equation}\label{niceweylsym}
a^{ijkl}+a^{iljk}+a^{ikjl}=0.
\end{equation} 

\hspace{6mm} In particular, taking $j=k=l$ in \eqref{niceweylsym}, we obtain
$$a^{ijjj}=a^{jijj}=0,$$
and for the case $i=j=l$, it then follows that
$$a^{iiki}+a^{iiik}+a^{ikii}=2a^{iiki}+0=2a^{iiki}=0.$$
Thus, we see that $a^{ijkl}=0$ whenever any integer occurs three or more times in the multi-index $ijkl$. Put another way, if--following \cite{Gi}--we denote by $\deg_m(I)$ the number of times an integer $m$ occurs in a multi-index $I$, we conclude that
\begin{equation}\label{weylcombinatorics0}
a^{ijkl}=0\text{ if }\deg_m(ijkl)\geq 3\text{ for some }0 \leq m \leq n-1.
\end{equation}
Moreover, taking $i=k$, $j=l$ in \eqref{niceweylsym}, we see that $a^{ijij}+a^{ijji}+a^{iijj}=2a^{ijij}+a^{iijj}=0$, so
\begin{equation}\label{weylcombinatorics00}
a^{iijj}=-2a^{ijij}.
\end{equation}
The full utility of this statement will become clear momentarily, but already we have the immediate consequence that $a^{iijj}=a^{jjii}$.

\hspace{6mm} Now, for any $A \in GL_n\mathbb{R}$, it's easy to see (by \eqref{gijks}) that $(\overline{x}^i):=(A^{-1})^i_jx^j$ is another normal coordinate system at $p$. If, moreover, $A \in O(1,n-1)$, then, letting $\overline{g}_{ij}$ denote the metric components with respect to $(\overline{x}^i)$, the transformation formula \eqref{gijs} gives us
$$\overline{g}_{ij}(p)=\frac{\partial x^a}{\partial \overline{x}^i}(p)\frac{\partial x^b}{\partial \overline{x}^j}(p)g_{ab}(p)=A^a_iA^b_j \eta_{ab}=\eta_{ij},$$
so $(\overline{x}^i)$ is Lorentz normal at $p$. In particular, since $\alpha^{ijkl}(w_{ij},w_{ijk})=\frac{\partial L}{\partial w_{ijkl}}(w_{ij},w_{ijk},w_{ijkl})$, it follows that 
$$\frac{\partial L}{\partial w_{ijkl}}(\overline{g}_{ij}(p),\overline{g}_{ij,k}(p),\overline{g}_{ij,kl}(p))=\frac{\partial L}{\partial w_{ijkl}}(g_{ij}(p),g_{ij,k}(p),g_{ij,kl}(p))=a^{ijkl},$$
and applying the tensoriality relation \eqref{scalardiffinvar1} of Lemma \ref{weyllemma} yields
\begin{equation}\label{lorentzinvar1}
a^{ijkl}=\frac{\partial x^i}{\partial \overline{x}^a}(p)\frac{\partial x^j}{\partial \overline{x}^b}(p)\frac{\partial x^k}{\partial \overline{x}^c}(p)\frac{\partial x^l}{\partial \overline{x}^d}(p)a^{abcd}=A^i_aA^j_bA^k_cA^l_da^{abcd}.
\end{equation}

\hspace{6mm} The Lorentz invariance \eqref{lorentzinvar1} will be tremendously valuable in narrowing the field of candidates for $L$. To begin exploiting this, consider the family of Lorentz transformations $A_m$ ($0 \leq m \leq n-1$) given on the standard basis $\{e_0,\ldots,e_{n-1}\}$ for $\mathbb{R}^{1,n-1}$ by $$A_me_m=-e_m,\text{ and }A_me_i=e_i\text{ for }i\neq m.$$
Applying \eqref{lorentzinvar1} to the transformations $A_m$ gives
$$a^{ijkl}=(-1)^{\deg_m(ijkl)}a^{ijkl}.$$
In particular, it follows that $a^{ijkl}=-a^{ijkl}=0$ whenever $\deg_m(ijkl)$ is odd for any $0 \leq m \leq n-1$, so 
\begin{equation}\label{weylcombinatorics1}
a^{rstu}=0\text{ unless }(rstu)=(iijj),\text{ }(ijij),\text{ or }(ijji)\text{ for some }0 \leq i\neq j \leq n-1
\end{equation}
(where the requirement that $i \neq j$ follows from the earlier statement \eqref{weylcombinatorics0}). Thus, we have
$$f_g(p)-b=\Sigma_{i \neq j}(a^{ijij}g_{ij,ij}(p)+a^{ijji}g_{ij,ji}(p)+a^{iijj}g_{ii,jj}(p))=\Sigma_{i \neq j}(2a^{ijij}g_{ij,ij}(p)+a^{iijj}g_{ii,jj}(p)),$$
and, by \eqref{weylcombinatorics00}, this gives us
\begin{equation}\label{weylfunc1}
f_g(p)-b=\Sigma_{i\neq j}(-a^{iijj}g_{ij,ij}(p)-a^{iijj}g_{ii,jj}(p))=-\Sigma_{i, j}a^{iijj}(g_{ij,ij}(p)-g_{ii,jj}(p)).
\end{equation}

\hspace{6mm} We can now easily dispatch with the case $n=2$: in this case, \eqref{weylfunc1} gives us 
$$f_g(p)-b=-a^{0011}(g_{01,01}(p)-g_{00,11}(p))-a^{1100}(g_{10,10}(p)-g_{11,00}(p)),$$
and  it follows immediately by comparison with \eqref{scalarcomp} that
$$f_g(p)=aR_g(p)+b,$$
where $a:=a^{0011}=a^{1100}$.

\hspace{6mm} Assume now that $n>2$. For any permutation $\sigma$ of $\{0,1,\ldots,n-1\}$ fixing $0$, the matrix $A_{\sigma} \in GL_n\mathbb{R}$ whose only nonzero entries are $A^i_{\sigma(i)}=1$ clearly gives an element of $O(n-1)\subset O(1,n-1)$. Applying \eqref{lorentzinvar1} to $A_{\sigma}$ yields
\begin{equation}\label{weylcombinatorics2}
a^{ijkl}=a^{\sigma(i)\sigma(j)\sigma(k)\sigma(l)},
\end{equation}
In particular, it follows that, whenever $n-1 \geq i \neq j>0$, we have $a^{ii00}=a^{00ii}=a^{0011}$, and $a^{iijj}=a^{1122}$. Making these substitutions in \eqref{weylfunc1} now gives us
\begin{equation}\label{weylfunc2}
f_g(p)-b=-a^{0011}\Sigma_j(g_{0j,0j}(p)-g_{00,jj}(p)+g_{j0,j0}(p)-g_{jj,00}(p))-a^{1122}\Sigma_{i,j>0}(g_{ij,ij}(p)-g_{ii,jj}(p)).
\end{equation}
(Note that the positive definite case would be complete at this point.)

\hspace{6mm} Finally, consider the prototypical boost $A\in O(1,n-1)$ given by $A_0^0=A_1^1=\sqrt{2}$, $A_0^1=A_1^0=1$, and $A_j^i=\delta_j^i$ outside of this $2\times 2$ block. Applying \eqref{lorentzinvar1} to this transformation and examining the term $a^{0022}$, we see that
$$a^{0022}=2a^{0022}+\sqrt{2}a^{1022}+\sqrt{2}a^{0122}+a^{1122}=2a^{0022}+a^{1122},$$
so $a^{1122}=-a^{0022}$. By \eqref{weylcombinatorics2}, it follows that $a^{0011}=-a^{1122}$ as well, so \eqref{weylfunc2} now gives us
\begin{eqnarray*}
f_g(p)-b &=& -a^{0011}\Sigma_{j>0}(g_{0j,0j}(p)-g_{00,jj}(p)+g_{j0,j0}(p)-g_{jj,00}(p))\\
&&+a^{0011}\Sigma_{i,j>0}(g_{ij,ij}(p)-g_{ii,jj}(p)))\\
&=& a^{0011}\Sigma_{i\neq j}\epsilon_i\epsilon_j(g_{ij,ij}(p)-g_{ii,jj}(p))\\
&=& a^{0011}\epsilon_i\epsilon_j(g_{ij,ij}(p)-g_{ii,jj}(p)).
\end{eqnarray*}
Thus, setting $a=a^{0011}$ and comparing the above relation with \eqref{scalarcomp}, we conclude that
$$f_g(p)=aR_g(p)+b,$$
as desired. 

\end{proof}

\begin{remark} To extend this proof to metrics of signature $(-1,\ldots,-1,1,\ldots,1)$, the only additional step necessary is to consider those elements of the associated isometry group $O(p,q)$ which permute timelike standard basis vectors, and arrive at the obvious analog of \eqref{weylcombinatorics2}.
\end{remark}

\hspace{6mm} If we were purely interested in identifying cousins of the Einstein-Hilbert action, then Theorem \ref{thm01} certainly gives us a good start. But the real appeal (historically) of the Einstein-Hilbert action lies in the fact that the associated Euler-Lagrange equations can be written in the form $E(g)=0$, where, for every Lorentz manifold $(M^n,g)$, $E(g)$ is a divergence-free $(2,0)$-tensor built out of curvature terms (namely, the Einstein tensor $G^{ij}:=Ric^{ij}-\frac{1}{2}Rg^{ij}$). An obvious next step, then, is to classify all tensors resembling the Einstein tensor. And here, we find a nice analog of the previous theorem, also due to Cartan and Weyl:

\begin{theorem}\label{weyltensorthm}\emph{(\cite{Weyl})} Let $T=(T^{ij}):W_n \to (\mathbb{R}^n)^{\otimes 2}$ be a smooth function (each real-valued component $T^{ij}$ is smooth) of the form
$$T^{ij}(w_{ij},w_{ijk},w_{ijkl})=\alpha^{ijklmp}(w_{ab},w_{abc})w_{klmp}+\beta^{ij}(w_{ab},w_{abc})$$
with the property that, for every Lorentz n-manifold $(M^n,g)$, $\exists$ a (2,0)-tensor field $E(g) \in \mathscr{T}_0^2(M)$ such that $T^{ij}(g_{ij},g_{ij,k},g_{ij,kl})=E^{ij}(g)$ in every coordinate chart $\xi: U\subseteq M \to \mathbb{R}^n$. Suppose, moreover, that $\alpha^{ijklmp}=\alpha^{jiklmp}$ and $\beta^{ij}=\beta^{ji}$ (so that $E(g)$ is symmetric). Then for every Lorentz $(M^n,g)$, $E^{ij}(g)=aRic^{ij}+(bR+c)g^{ij}$ for some constants $a,b,c \in \mathbb{R}$. Furthermore, if the tensor $E(g)$ is divergence free for each $(M^n,g)$, it follows that $E(g)$ is a linear combination of the Einstein tensor $G^{ij}$ and the (inverse) metric $g^{ij}$.
\end{theorem}

\begin{proof} The proof is a modification of our proof of Theorem \ref{thm01}. First, recall (from the expressions we obtained for the Riemann curvature tensor in the proof of the previous theorem) that, in a Lorentz normal coordinate system at a point $p\in M$, the Ricci tensor $Ric^{ij}(p)$ of a Lorentz manifold $(M^n,g)$ is given by
\begin{eqnarray*}
Ric^{ij}(p)&=&\eta_{ir}\eta_{js}\frac{1}{2}\eta_{ka}(g_{sa,rk}+g_{ra,sk}+g_{sk,ar}-g_{sr,ak}-g_{sa,kr}-g_{ka,sr})(p)\\
&=&\frac{1}{2}\epsilon_i\epsilon_j\epsilon_k(g_{jk,ik}+g_{ik,jk}+g_{jk,ki}-g_{ji,kk}-g_{jk,ki}-g_{kk,ji})(p)\\
&=&\frac{1}{2}\epsilon_i\epsilon_j\epsilon_k(g_{ik,jk}+g_{jk,ik}-g_{ij,kk}-g_{kk,ij})(p),
\end{eqnarray*}
and, by \eqref{scalarcomp},
$$R(p)g^{ij}(p)=\epsilon_k\epsilon_l(g_{kl,kl}-g_{kk,ll})(p)\eta_{ij}.$$

\hspace{6mm} As in the scalar case, we can assume without loss of generality that the derivatives of $T$ satisfy the symmetries
\begin{equation}\label{weyltenssym}
\frac{\partial T^{ab}}{\partial w_{ij}}=\frac{\partial T^{ab}}{\partial w_{ji}},\text{ }\frac{\partial T^{ab}}{\partial w_{ijk}}=\frac{\partial T^{ab}}{\partial w_{jik}},\text{ }\frac{\partial T^{ab}}{\partial w_{ijkl}}=\frac{\partial T^{ab}}{\partial w_{jikl}}=\frac{\partial T^{ab}}{\partial w_{ijlk}}=\frac{\partial T^{ab}}{\partial w_{jilk}}.
\end{equation}
And by an argument identical to the proof of Lemma \ref{weyllemma}, it's easy to see that, for any Lorentz $n$-fold $(M,g)$, on the overlap of any two coordinate charts $(x^i)$ and $(\tilde{x}^i)$, we have
\begin{equation}\label{weyltensderiv}
\frac{\partial T^{ij}}{\partial w_{klmp}}(g_{ij},g_{ij,k},g_{ij,kl})=\frac{\partial x^i}{\partial \tilde{x}^a}\frac{\partial x^j}{\partial\tilde{x}^b}\frac{\partial x^k}{\partial \tilde{x}^c}\frac{\partial x^l}{\partial \tilde{x}^d}\frac{\partial x^m}{\partial \tilde{x}^e}\frac{\partial x^p}{\partial \tilde{x}^f}\frac{\partial T^{ab}}{\partial w_{cdef}}(\tilde{g}_{ij},\tilde{g}_{ij,k},\tilde{g}_{ij,kl}),
\end{equation}
(since $T^{ij}(g_{ij},g_{ij,k},g_{ij,kl})=E^{ij}(g)=\frac{\partial x^i}{\partial \tilde{x}^a}\frac{\partial x^j}{\partial \tilde{x}^b}\tilde{E}^{ab}(g)=\frac{\partial x^i}{\partial \tilde{x}^b}\frac{\partial x^j}{\partial \tilde{x}^b}T^{ab}(\tilde{g}_{ij},\tilde{g}_{ij,k},\tilde{g}_{ij,kl})$).

\hspace{6mm} Now, set $a^{ijklmp}=\alpha^{ijklmp}(\eta_{ab},0)$ and $b^{ij}=\beta^{ij}(\eta_{ab},0)$, so that 
$$E^{ij}(g)(p)=a^{ijklmp}g_{kl,mp}(p)+b^{ij}$$
in Lorentz normal coordinates at a point $p$ of any Lorentz manifold $(M^n,g)$. Then
\begin{equation}\label{weyltenssym2}
a^{ijklmp}=a^{ijlkmp}=a^{ijklpm}=a^{ijlkpm},
\end{equation}
by \eqref{weyltenssym}, and for any $A \in O(1,n-1)$, it follows from \eqref{weyltensderiv} that
\begin{equation}\label{wtenslorentzinvar}
A_r^iA_s^jA_t^kA_u^lA_v^mA_w^pa^{rstuvw}=a^{ijklmp}.
\end{equation} 
Moreover, by the same argument we used in the proof of Theorem \ref{thm01} to obtain \eqref{niceweylsym}, we again have
\begin{equation}
a^{ijklmp}+a^{ijkmlp}+a^{ijkpml}+a^{ijklpm}+a^{ijkmpl}+a^{ijkplm}=0,
\end{equation}
\begin{equation}\label{weylpermsym}
\text{so (by \eqref{weyltenssym}) }a^{ijklmp}+a^{ijkmpl}+a^{ijkplm}=0,
\end{equation}
and consequently
\begin{equation}\label{weylpermsym1}
a^{ijklmp}=0\text{ if }\deg_m(klmp)>2\text{ for any }0\leq m \leq n-1,
\end{equation} 
as before.

\hspace{6mm} As in the proof of Theorem \ref{thm01}, we can use the Lorentz invariance \eqref{wtenslorentzinvar} of the terms $a^{ijklmp}$ to conclude that
\begin{equation}\label{wtensevendegree}
a^{ijklmp}=0\text{ whenever }\deg_m(ijklmp)\text{ is odd for any }0 \leq m \leq n-1.
\end{equation}
When $i \neq j$, this implies that $a^{ijklmp}=0$ unless  
\begin{eqnarray*}
(klmp)&=&(ijrr), (irjr), (irrj), (jirr),\\
&& (rijr), (rirj), (jrir), (rjir),\\
&& (rrij), (jrri), (rjri),\text{ or } (rrji)
\end{eqnarray*}
for some $0 \leq r \leq n-1$; furthermore \eqref{weylpermsym1} implies that $a^{ijklmp}$ will vanish if this $r=i$ or $j$ (so immediately we see that $a^{ijklmp}=0$ whenever $i \neq j$ in the case $n=2$). By the symmetries \eqref{weyltenssym2}, we already know that $a^{ijijkk}=a^{ijjikk}$, $a^{ijkkij}=a^{ijkkji}$, $a^{ijikjk}=a^{ijikkj}=a^{ijkijk}=a^{ijkikj}$, and $a^{ijjkik}=a^{ijkjik}=a^{ijjkki}=a^{ijkjki}$. And from the symmetries \eqref{weylpermsym}, we see that
$$a^{ijijkk}+2a^{ijikjk}=a^{ijijkk}+a^{ijikjk}+a^{ijikkj}=0,$$
so $2a^{ijikjk}=-a^{ijijkk}=-a^{ijjikk}=2a^{ijjkik}$, and consequently
$$a^{ijkkij}-a^{ijijkk}=a^{ijkkij}+2a^{ijikjk}=a^{ijkkij}+a^{ijikjk}+a^{ijjkik}=a^{ijkkij}+a^{ijkijk}+a^{ijkjki}=0,$$
by \eqref{weylpermsym}. We've now shown that $a^{ijkkij}=a^{ijijkk}=-2a^{ijikjk}=-2a^{ijjkik},$ and combining this with the other properties we've derived so far, we see already that
\begin{eqnarray*}
a^{ijklmp}g_{kl,mp}(p)&=&\Sigma_{k\neq i,j}(2a^{ijijkk}g_{ij,kk}+2a^{ijkkij}g_{kk,ij}+4a^{ijikjk}g_{ik,jk}+4a^{ijjkik}g_{jk,ik})(p)\\
&=&-2\Sigma_{k \neq i,j}a^{ijijkk}(g_{ik,jk}+g_{jk,ik}-g_{ij,kk}-g_{kk,ij})(p).
\end{eqnarray*}
--a promising form for something we wish to show is a component of the Ricci tensor.

\hspace{6mm} Next, by applying the Lorentz invariance \eqref{wtenslorentzinvar} to the permutations $A_{\sigma}\in S_{n-1}\subset O(1,n-1)$ and the prototypical boost from the proof of Theorem \ref{thm01}, it's easy to check that (still with the stipulation that $i \neq j$), for all $0 \leq k,l \leq n-1$ different from $i$ and $j$, we have $a^{ijijkk}=a^{ijijll}$ when $k,l>0$, and $a^{ijijkk}=-a^{ijij00}$ when $k$ (as well as $i,j$) is positive. Hence,
$$a^{ijklmp}g_{kl,mp}(p)=2a^{ijij00}\Sigma_{k\neq i,j}(g_{ik,jk}+g_{jk,ik}-g_{ij,kk}-g_{kk,ij})(p)$$
when $i,j>0$, and
$$a^{0jklmp}g_{kl,mp}(p)=-a^{0j0jrr}\Sigma_{k\neq 0,j}(g_{0k,jk}+g_{jk,0k}-g_{0j,kk}-g_{kk,0j})(p)$$
when $j>0$ and $r\in \{1,\ldots,n-1\}$ is any index different from $0$ and $j$.

By another application of the same Lorentz invariance arguments, we observe, moreover, that $a^{ijij00}=a^{klkl00}$ when $i,j,k,l>0$, and $a^{020211}=a^{121200}$. Thus, letting $a=4a^{020211}$, we can indeed conclude that
\begin{equation}\label{ricciobs}
a^{ijklmp}g_{kl,mp}(p)=a\epsilon_i\epsilon_j\Sigma_{k \neq i,j}\epsilon_k(g_{ik,jk}+g_{jk,ik}-g_{ij,kk}-g_{kk,ij})(p)=aRic^{ij}(p)
\end{equation}
when $i \neq j$.

\hspace{6mm} Now, let $b^{ijklmp}$ denote the unique constants satisfying the symmetries \eqref{weyltenssym2} such that 
$$aRic^{ij}(p)=b^{ijklmp}g_{kl,mp}(p)$$
in any Lorentz normal coordinate system at a point $p$ of any Lorentz manifold $(M^n,g)$. Setting $c^{ijklmp}:=a^{ijklmp}-b^{ijklmp}$, we note that the terms $c^{ijklmp}$ satisfy the same symmetries as $a^{ijklmp}$, including the Lorentz invariance \eqref{wtenslorentzinvar}; and, by \eqref{ricciobs}, it's clear that $c^{ijklmp}=0$ whenever $i \neq j$.

\hspace{6mm} For simplicity, assume that $n>2$ for the remainder of the argument (as usual, the $n=2$ case can be dispensed with fairly easily, so we'll leave it to the reader). Since $c^{ijklmp}$ satisfies both \eqref{wtenslorentzinvar} and \eqref{weylpermsym1}, now-familiar arguments show that the only nonvanishing terms $c^{iiklmp}$ are those of the forms $c^{iijjkk}$ or $c^{iijkjk}=c^{iijkkj},$ where $j\neq k,$
$$c^{iijjkk}=-2c^{iijkjk}=-2c^{iikjkj}=c^{iikkjj},$$
$$c^{iijjkk}=c^{\sigma(i)\sigma(i)\sigma(j)\sigma(j)\sigma(k)\sigma(k)}\text{ for any automorphism }\sigma\text{ of }\{0,1,\ldots,n-1\}\text{ fixing }0,$$
$$\text{ and }c^{iijjkk}=-c^{iijj00}\text{ if }i,j,\text{ and }k\text{ are all distinct and positive}.$$
(Likewise, $c^{iijjkk}=-c^{00jjkk}$ when $i,j,k\in \{0,1,\ldots,n-1\}$ are distinct and positive). Consequently, we see that
\begin{eqnarray*}
c^{iiklmp}g_{kl,mp}(p)&=&c^{iijjkk}(g_{jj,kk}-g_{jk,jk})(p)\\
&=&c^{111122}\Sigma_{k\neq i}\epsilon_k(g_{ii,kk}-g_{ik,ik}+g_{kk,ii}-g_{ik,ik})(p)-c^{110022}\Sigma_{j,k\neq i}\epsilon_j\epsilon_k(g_{jj,kk}-g_{jk,jk})(p)
\end{eqnarray*}
when $i>0$, and 
$$c^{00klmp}g_{kl,mp}(p)=c^{000022}\Sigma_{k>0}(g_{00,kk}-g_{0k,0k}+g_{kk,00}-g_{0k,0k})(p)+c^{001122}\Sigma_{j,k>0}\epsilon_j\epsilon_k(g_{jj,kk}-g_{jk,jk})(p).$$

Finally, by yet another application of the invariance of $c^{ijklmp}$ under the action of the boost from the proof Theorem \ref{thm01}, we obtain
\begin{equation}\label{wtenscomp1}
c^{000022}=4c^{000022}+2c^{001122}+2c^{110022}+c^{111122},
\end{equation}
\begin{equation}\label{wtenscomp2}
c^{111122}=4c^{111122}+2c^{110022}+2c^{001122}+c^{000022},
\end{equation}
\begin{equation}\label{wtenscomp3}
c^{001122}=4c^{001122}+2c^{000022}+2c^{111122}+c^{110022},
\end{equation}
\begin{equation}\label{wtenscomp4}
\text{and }c^{110022}=4c^{110022}+2c^{111122}+2c^{000022}+c^{001122}.
\end{equation}
Taking the difference of \eqref{wtenscomp1} and \eqref{wtenscomp2} yields $c^{000022}=c^{111122}$, and it follows similarly from \eqref{wtenscomp3} and \eqref{wtenscomp4} that $c^{110022}=c^{001122}$. Finally, applying these equalities to \eqref{wtenscomp1}, we see that $c^{000022}=-c^{001122}$ as well, so 
$$c^{iiklmp}g_{kl,mp}(p)=c^{111122}\epsilon_j\epsilon_k(g_{jj,kk}-g_{jk,jk})(p)\text{ when }i>0,$$
$$\text{and }c^{00klmp}g_{kl,mp}(p)=-c^{111122}\epsilon_j\epsilon_k(g_{jj,kk}-g_{jk,jk})(p).$$
Thus, setting $b=-c^{111122}$, we have
$$c^{ijklmp}g_{kl,mp}(p)=b\epsilon_k\epsilon_l(g_{kl,kl}-g_{kk,ll})(p)\eta_{ij}=bR(p)g^{ij}(p).$$

\hspace{6mm} Since we've now shown that $a^{ijklmp}g_{kl,mp}(p)=aRic^{ij}(p)+bR(p)g^{ij}(p)$ in normal coordinates, it follows readily that the terms $b^{ij}$ satisfy the Lorentz invariance $A^i_rA^j_sb^{rs}=b^{ij}$ ($A \in O(1,n-1)$), and a quick application of the usual arguments yields $b^{ij}=b^{11}\eta_{ij}=cg^{ij}(p)$ (where $c:=b^{11}$). Thus,
$$a^{ijklmp}g_{kl,mp}(p)+b^{ij}=aRic^{ij}(p)+bR(p)g^{ij}(p)+cg^{ij}(p)$$
in any normal coordinate system, so indeed,
$$E^{ij}(g)=aRic^{ij}+bRg^{ij}+cg^{ij},$$
as desired.

\hspace{6mm} That $E(g)$ is divergence-free (for every $g$) precisely when it is a linear combination of the metric and the Einstein tensor $G=Ric-\frac{1}{2}Rg$ now follows from the trivial observation that $Rg$ is divergence-free precisely when $R$ is constant.
\end{proof}

\hspace{6mm} When $n \leq 4,$ a combinatorial argument due to Lovelock reveals that the conclusion of Theorem \ref{weyltensorthm} still holds if we replace the linearity assumption $\frac{\partial^2T^{ij}}{\partial w_{ijkl}\partial w_{mpqr}}=0$ with the requirement that $E(g)$ be divergence-free (see Chapter 8 of \cite{LR}). This works because the divergence-free requirement introduces more symmetries to the derivatives of $T^{ij}$, which together imply $\frac{\partial^2T^{ij}}{\partial w_{ijkl}\partial w_{mpqr}}=0$ in dimension $<5$, via the pigeonhole principle. Hence, in the dimensions of interest in General Relativity, the only divergence-free tensors that are functions of the metric components and their first and second derivatives are those of the form $aG+bg$.

\section{Low-Order Lagrangians of the Metric and a Matter Field}\label{sec:new}

\hspace{6mm} With an understanding of the classical results, we're now prepared to prove our main theorem, from which the conclusion of Bray's conjecture follows easily.

\hspace{6mm} We now take the domain of our Lagrangians to be $Y^m_n=W_n\times(\mathbb{R}^n)^{\otimes m}\times (\mathbb{R}^n)^{\otimes (m+1)}$, where $W_n$ is still given by \eqref{wdef}. For a function $L:Y^m_n\to \mathbb{R}$, the statement that $L$ has the algebraic form \eqref{danhyp} can be stated more clearly as follows:
\begin{equation}\label{algconstraint}
D_uD_vD_wL=0\text{ }\forall u,v,w \in 0\times(\mathbb{R}^n)^{\otimes 3}\times(\mathbb{R}^n)^{\otimes 4}\times (\mathbb{R}^n)^{\otimes m}\times(\mathbb{R}^n)^{\otimes m+1} \subset Y_n^m.
\end{equation}
With these definitions in place, we are now in a position to state our result:

\begin{theorem}\label{myscalarthm}\emph{(Main Theorem)} Let $L:Y^m_n \to \mathbb{R}$ be a smooth function satisfying \eqref{algconstraint} and suppose that, for every triple $(M^n,g,D)$ (where $g$ is a Lorentz metric and $D$ is a type (0,m)-tensor on $M$), $\exists$ $f_{g,D}\in C^{\infty}(M)$ such that
$$L(g_{ij},g_{ij,k},g_{ij,kl},D_{i_1\cdots i_m},D_{i_1\cdots i_m,k})=f_{g,D}$$
in all coordinate systems on $M$. Then there are constants $a,b,c \in \mathbb{R}$, a quadratic invariant $Q_g(D)=\mu^{I J}(g_{ab})D_ID_J$, an invariant trace term $T_g(D)$, and a divergence term $Bd_g(D)$ such that 
$$f_{g,D}=a+bR_g+c|d\gamma|_g^2+Q_g(D)+T_g(D)+Bd_g(D)$$
for all triples $(M,g,D)$ (where $d\gamma$ is the exterior derivative of the m-form $\gamma_{i_1\cdots i_m}=\frac{1}{m!}\Sigma_{\sigma \in S_m}sgn(\sigma)D_{i_{\sigma(1)}\cdots i_{\sigma(m)}}$).
\end{theorem}

\vspace{4mm}

\begin{remark} Once again, by replacing $L$ with the function $L\circ \phi$, where $\phi: Y_n^m \to Y_n^m$ is given by 
$$\phi(x_{ij},x_{ijk},x_{ijkl},y_{I},y_{I,j})=(\frac{1}{2}(x_{ij}+x_{ji}),\frac{1}{2}(x_{ijk}+x_{jik}),\frac{1}{4}(x_{ijkl}+x_{jikl}+x_{ijlk}+x_{jilk}),y_I,y_{I,j}),$$
we can assume without loss of generality that $L$ satisfies the obvious symmetries 
\begin{equation}\label{mythmsyms}
D_vL=D_{\phi(v)}L;\text{ }D_vD_uL=D_{\phi(v)}D_{\phi(u)}L\text{ for }u,v \in Y_n^m.
\end{equation}
\end{remark}

\begin{remark} Note that the result extends immediately to the case where $D$ takes values in type $(r,s)$-tensor fields, since the raising and lowering of indices is a $0$th-order operation in the metric.
\end{remark}

Once we've established the main theorem, it's not difficult to see how the result of Bray's conjecture follows:

\begin{corollary}\label{brayconj}\emph{(Bray's Conjecture)} Let $L:Y^m_n \to \mathbb{R}$ be a smooth function satisfying \eqref{algconstraint} such that, for every triple $(M^n,g,\nabla)$ ($\nabla$ an affine connection on $TM$), $\exists$ $f_{g,\nabla} \in C^{\infty}(M)$ such that
$$L(g_{ij},g_{ij,k},g_{ij,kl},\Gamma_{ijk},\Gamma_{ijk,l})=f_{g,\nabla}$$
in all coordinates on $M$. Then there are constants $a,b,c \in \mathbb{R}$, a quadratic invariant $Q_g(D)=\mu^{IJ}(g_{ab})D_ID_J$, and a divergence term $Bd_g(D)$ such that
$$f_{g,\nabla}=a+bR_g+c|d\gamma|_g^2+Q_g(D)+Bd_g(D)$$
for all triples $(M,g,\nabla)$ (where $D_{ijk}=\Gamma_{ijk}-\frac{1}{2}(g_{jk,i}+g_{ik,j}-g_{ij,k})$ and $\gamma$ is the fully antisymmetric part of $D$).
\end{corollary}

\begin{proof}[Proof of Corollary \ref{brayconj} (Assuming Theorem \ref{myscalarthm})] Given such a function $L$, define another function $\tilde{L}:Y_n^m\to \mathbb{R}$ by 
$$\tilde{L}(x_{ij},x_{ijk},x_{ijkl},y_{ijk},y_{ijkl})=L(x_{ij},x_{ijk},x_{ijkl},y_{ijk}+\frac{1}{2}(x_{jki}+x_{ikj}-x_{ijk}),y_{ijkl}+\frac{1}{2}(x_{jkil}+x_{ikjl}-x_{ijkl})),$$
so that
$$\tilde{L}(g_{ij},g_{ij,k},g_{ij,kl},D_{ijk},D_{ijk,l})=L(g_{ij},g_{ij,k},g_{ij,kl},\Gamma_{ijk},\Gamma_{ijk,l})=f_{g,\nabla},$$
where $\nabla$ is the connection that differs from the Levi-Civita connection $\nabla_g$ by the tensor $D$. $\tilde{L}$ then clearly satisfies the hypotheses of Theorem \ref{myscalarthm} in the case $m=3$ with $f_{g,D}=f_{g,\nabla},$ so $f_{g,\nabla}$ has the form
$$f_{g,\nabla}=a+bR_g+c|d\gamma|_g^2+Q_g(D)+T_g(D)+Bd_g(D).$$
To complete the proof, observe that, since the trace component $T_g(D)=\alpha^{ijk}(g_{ab})D_{ijk}$ is invariant under changes of coordinates, the coordinate transformation $x^i\mapsto -x^i$ yields
$$\alpha^{ijk}(g_{ab})D_{ijk}=\alpha^{ijk}(g_{ab})(-D_{ijk})=-\alpha^{ijk}(g_{ab})D_{ijk},$$
so $T_g(D)$ must vanish identically.
\end{proof}

\begin{proof}[Proof of Theorem \ref{myscalarthm}] If we ignore for a moment the summands of $L$ that depend nontrivially on the tensor $D$, the statement of Theorem \ref{myscalarthm} strongly resembles that of Theorem \ref{thm01}. Indeed, defining a function $\hat{L}: W_n \to \mathbb{R}$ by 
$$\hat{L}(x_{ij},x_{ijk},x_{ijkl})=L(x_{ij},x_{ijk},x_{ijkl},0,0),$$
we see that, since $\hat{L}(g_{ij},g_{ij,k},g_{ij,kl})=f_{g,0}$ holds in all coordinates on every Lorentz manifold, the function $\hat{L}$ satisfies nearly all the hypotheses of Theorem \ref{thm01}, the only exception being that $\hat{L}$ is allowed to depend quadratically on the second derivatives of the metric. Thus, to show that $f_{g,0}$ has the desired form, we simply need to show that $\frac{\partial^2 \hat{L}}{\partial x_{ijkl}\partial x_{qrst}}=0$.

\hspace{6mm} By \eqref{algconstraint}, we know that
$$\frac{\partial\hat{L}}{\partial x_{ijkl}}(g_{ab},g_{ab,c},g_{ab,cd})=\alpha^{ijkl}(g_{ab})+\beta^{ijklqrs}(g_{ab})g_{qr,s}+\gamma^{ijklqrst}(g_{ab})g_{qr,st},$$
where $\beta^{ijklqrs}(g_{ab})=\frac{\partial^2 \hat{L}}{\partial x_{ijkl}\partial x_{qrs}}(g_{ab},g_{ab,c},g_{ab,cd})$ and $\gamma^{ijklqrst}(g_{ab})=\frac{\partial^2\hat{L}}{\partial x_{ijkl}\partial x_{qrst}}(g_{ab},g_{ab,c},g_{ab,cd}).$
By Lemma \ref{weyllemma}, the terms $\frac{\partial \hat{L}}{\partial x_{ijkl}}(g_{ab},g_{ab,c},g_{ab,cd})$ form the components of a type (4,0) tensor field. An immediate consequence of this tensoriality is the fact that the terms $\frac{\partial \hat{L}}{\partial x_{ijkl}}(g_{ab},g_{ab,c},g_{ab,cd})$ are unchanged by the change of coordinates $(x^i)\to -(x^i)$, from which it follows that $\beta^{ijklqrs}(g_{ab})g_{qr,s}=-\beta^{ijklqrs}(g_{ab})g_{qr,s}=0,$ and, consequently,
$$\frac{\partial \hat{L}}{\partial x_{ijkl}}(g_{ab},g_{ab,c},g_{ab,cd})=\alpha^{ijkl}(g_{ab})+\gamma^{ijklqrst}(g_{ab})g_{qr,st}.$$

\hspace{6mm} Now, let $p$ be a point on an arbitrary Lorentz manifold $(M^n,g)$, and let $(x^i)$ be a Lorentz normal coordinate system centered at $p$. Given $(b^i_{jk})\in (\mathbb{R}^n)^{\otimes 3}$ satisfying $b^i_{jk}=b^i_{kj}$ and $(c^i_{jkl})\in (\mathbb{R}^n)^{\otimes 4}$ satisfying $c^i_{j_{\sigma(1)}j_{\sigma(2)}j_{\sigma(3)}}=c^i_{j_1j_2j_3}$ for all permuations $\sigma \in S_3$, let $(\tilde{x}^i)$ be a smooth coordinate system with $\tilde{x}^i(p)=0$, defined implicitly on a neighborhood of $p$ by 
\begin{equation}\label{gencoordchange}
x^i=\tilde{x}^i+\frac{1}{2}b^i_{jk}\tilde{x}^j\tilde{x}^k+\frac{1}{6}\eta_{ia}c^a_{jkl}\tilde{x}^j\tilde{x}^k\tilde{x}^l.
\end{equation}
(Such a coordinate system always exists, of course, by the inverse function theorem.) Then clearly
$$\frac{\partial x^i}{\partial \tilde{x}^a}(p)=\delta^i_a, \text{  }\frac{\partial^2 x^i}{\partial \tilde{x}^a\partial\tilde{x}^b}(p)=b^i_{ab},\text{ and }\frac{\partial^3x^i}{\partial \tilde{x}^a\partial\tilde{x}^b\partial\tilde{x}^c}(p)=\eta_{ir}c^r_{abc},$$
so we see that
$$\tilde{g}_{ij}(p)=g_{ij}(p)=\eta_{ij},\text{ and }\tilde{g}_{ij,kl}(p)=g_{ij,kl}(p)+(b^a_{ik}b^b_{jl}+b^a_{il}b^b_{jk})\eta_{ab}+c^i_{jkl}+c^j_{ikl},$$
and the tensoriality of $\frac{\partial \hat{L}}{\partial x_{ijkl}}(g_{ab},g_{ab,c},g_{ab,cd})$ implies
\begin{eqnarray*}
0&=&\frac{\partial\hat{L}}{\partial x_{ijkl}}(g_{ab},g_{ab,c},g_{ab,cd})(p)-\frac{\partial\hat{L}}{\partial x_{ijkl}}(\tilde{g}_{ab},\tilde{g}_{ab,c},\tilde{g}_{ab,cd})(p)\\
&=&\gamma^{ijklqrst}(\eta_{ab})(\eta_{ab}b^a_{qs}b^b_{rt}+\eta_{ab}b^a_{qt}b^b_{rs}+c^q_{rst}+c^r_{qst}).
\end{eqnarray*}
Applying the symmetries of the terms $\gamma^{ijklqrst}$ (corresponding to \eqref{mythmsyms}) to the above relation, we obtain 
\begin{equation}\label{gammasyms1}
\gamma^{ijklqrst}(\eta_{ab})(\eta_{ab}b^a_{qs}b^b_{rt}+c^q_{rst})=0.
\end{equation}

\hspace{6mm} Now, fix some indices $0\leq q,r,s,t<n$; take $(b^p_{qr})=0$, $c^q_{rst}=c^q_{srt}=c^q_{rts}=c^q_{tsr}=c^q_{str}=c^q_{trs}=1$, and set $c^i_{jkl}=0$ for all other choices of $0 \leq i,j,k,l<n$. In this case, \eqref{gammasyms1}, together with the other (\eqref{mythmsyms}) symmetries of the terms $\gamma^{ijklqrst}$, yields 
\begin{equation}\label{gammasyms2}
\gamma^{ijklqrst}(\eta_{ab})+\gamma^{ijklqstr}(\eta_{ab})+\gamma^{ijklqtrs}(\eta_{ab})=0.
\end{equation}

\hspace{6mm} Next, fix some $0 \leq u,v<n$, set $(c^i_{jkl})=0$, $b^1_{uv}=b^1_{vu}=1$, and $b^i_{jk}=0$ for all other $i,j,k$. Now \eqref{gammasyms1} gives us
\begin{equation}\label{moregammasyms}
\gamma^{ijkluuvv}(\eta_{ab})+\gamma^{ijkluvvu}(\eta_{ab})+\gamma^{ijklvuuv}(\eta_{ab})+\gamma^{ijklvvuu}(\eta_{ab})=0,
\end{equation}
for all $0 \leq u,v <n$.

\hspace{6mm} Finally, given $0 \leq m,q,r,s,t <n$, setting $b^m_{qr}=b^m_{rq}=b^m_{st}=b^m_{ts}=1$, and letting all other $b^i_{jk}$ and $c^i_{jkl}=0$ in \eqref{gammasyms1}, we obtain
\begin{eqnarray*}
0&=&\gamma^{ijklqqrr}(\eta_{ab})+\gamma^{ijklqrrq}(\eta_{ab})+\gamma^{ijklrqqr}(\eta_{ab})+\gamma^{ijklrrqq}(\eta_{ab})\\
&&+\gamma^{ijklsstt}(\eta_{ab})+\gamma^{ijklstts}(\eta_{ab})+\gamma^{ijkltsst}(\eta_{ab})+\gamma^{ijklttss}(\eta_{ab})\\
&&+\gamma^{ijklqsrt}(\eta_{ab})+\gamma^{ijklqtrs}(\eta_{ab})+\gamma^{ijklrsqt}(\eta_{ab})+\gamma^{ijklrtqs}(\eta_{ab})\\
&&+\gamma^{ijklsqtr}(\eta_{ab})+\gamma^{ijklsrtq}(\eta_{ab})+\gamma^{ijkltqsr}(\eta_{ab})+\gamma^{ijkltrsq}(\eta_{ab}).
\end{eqnarray*}
By \eqref{moregammasyms}, the first two lines of the above relation vanish, and since $\gamma^{ijklqsrt}=\gamma^{ijklsqtr}$, $\gamma^{ijklqtrs}=\gamma^{ijkltqsr}$, and so on, by \eqref{mythmsyms}, the equation above reduces to
\begin{equation}\label{almostgpart}
2(\gamma^{ijklqsrt}(\eta_{ab})+\gamma^{ijklqtrs}(\eta_{ab})+\gamma^{ijklrsqt}(\eta_{ab})+\gamma^{ijklrtqs}(\eta_{ab}))=0.\end{equation}
Applying \eqref{gammasyms2} to the first two and last two summands of \eqref{almostgpart} now yields
$$0=2(-\gamma^{ijklqrst}(\eta_{ab})-\gamma^{ijklrqts}(\eta_{ab}))=-4\gamma^{ijklqrst}(\eta_{ab}),$$
so we see that $\gamma^{ijklqrst}(\eta_{ab})=0$, and by the tensoriality of the terms $\gamma^{ijklqrst}(g_{ab})$, it follows that
$$\frac{\partial^2\hat{L}}{\partial x_{ijkl}\partial x_{qrst}}(g_{ab},g_{ab,c},g_{ab,cd})=\gamma^{ijklqrst}(g_{ab})=0$$
in all coordinate charts on every Lorentz manifold. Thus, we can apply Theorem \ref{thm01} to the function $\hat{L}$, and conclude that
\begin{equation}\label{metricpart}
f_{g,0}=a+bR_g
\end{equation}
for all Lorentz manifolds $(M^n,g)$.

\begin{remark} Note the central role that the assumption $\frac{\partial ^3 \hat{L}}{\partial x_{ijkl}\partial x_{mpqr}\partial x_{tuv}}=0$ (a consequence of \eqref{algconstraint}) plays in our proof of the statement $f_{g,0}=aR_g+b$. By removing this assumption, we would allow the terms $\gamma^{ijklqrst}$ to change under general coordinate transformations of the form \eqref{gencoordchange}, causing us to lose the symmetry \eqref{moregammasyms}, and allowing $R^2, |Riem|^2,$ and other quadratic curvature terms to appear in $f_{g,0}$. 
\end{remark}

\hspace{6mm} Our next goal will be to characterize the terms $\alpha^{Ij}=\frac{\partial L}{\partial y_{I,j}}$. To begin, we'll employ an analog of Lemma \ref{weyllemma} (whose proof is identical to that of Lemma \ref{weyllemma}--if not slightly easier, since we don't have to symmetrize) to conclude that, for every triple $(M^n,g,D)$, the coordinate expressions $\alpha^{Ij}(g_{ij},g_{ij,k},g_{ij,kl},D_I,D_{I,j})$ form the components of a type-$(m+1,0)$ tensor field.

\hspace{6mm} By \eqref{algconstraint}, we know that $\alpha^{Ij}$ must have the form
\begin{eqnarray*}\alpha^{Ij}(g_{ij},g_{ij,k},g_{ij,kl},D_I,D_{I,j})&=&\beta^{Ij}(g_{ab})+\beta^{Ijlmp}(g_{ab})g_{lm,p}+\beta^{Ijlmpq}(g_{ab})g_{lm,pq}\\
&&+\eta^{IjK}(g_{ab})D_K+\eta^{IjKl}(g_{ab})D_{K,l}.
\end{eqnarray*}
We wish to show that, in fact, the only nontrivial terms above are $\beta^{Ij}(g_{ab})$ and $\eta^{IjKl}(g_{ab})D_{K,l}$. 

Let $(M,g,D)$ be a triple consisting of a manifold, a Lorentz metric, and a $(0,m)$-tensor field, and let $\xi$ be an arbitrary coordinate chart on $M$. Set $\tilde{\xi}=-\xi$, so that the tensoriality of $\alpha^{Ij}$ yields
 \begin{eqnarray*}\alpha^{Ij}(\tilde{g}_{ij},\tilde{g}_{ij,k},\tilde{g}_{ij,kl},\tilde{D}_I,\tilde{D}_{I,j})&=&\beta^{Ij}(g_{ab})-\beta^{Ijlmp}(g_{ab})g_{lm,p}+\beta^{Ijlmpq}(g_{ab})g_{lm,pq}\\
&&+(-1)^m\eta^{IjK}(g_{ab})D_K+(-1)^{m+1}\eta^{IjKl}(g_{ab})D_{K,l}\\
=(-1)^{m+1}\alpha^{Ij}(g_{ij},g_{ij,k},g_{ij,kl},D_I,D_{I,j})&=&(-1)^{m+1}(\beta^{Ij}(g_{ab})+\beta^{Ijlmp}(g_{ab})g_{lm,p}+\beta^{Ijlmpq}(g_{ab})g_{lm,pq}\\
&&+\eta^{IjK}(g_{ab})D_K+\eta^{IjKl}(g_{ab})D_{K,l}).
\end{eqnarray*}
If $m$ is even, it follows that
$\beta^{Ij}(g_{ab})+\beta^{Ijlmpq}(g_{ab})g_{lm,pq}+\eta^{IjK}(g_{ab})D_K=0,$
and, consequently
\begin{equation}\label{alphameven}
\alpha^{Ij}(g_{ij},g_{ij,k},g_{ij,kl},D_I,D_{I,j})=\beta^{Ijlmp}(g_{ab})g_{lm,p}+\eta^{IjKl}(g_{ab})D_{K,l}.
\end{equation}
When $D=0$, this gives us $\alpha^{Ij}(g_{ij},g_{ij,k},g_{ij,kl},0,0)=\beta^{Ijlmp}(g_{ab})g_{lm,p}$; hence, the terms $\beta^{Ijlmp}(g_{ab})g_{lm,p}$ form the components of an $(m+1,0)$-tensor for all $(M,g)$, and since we can choose coordinates at every point for which $g_{lm,p}=0$, it follows that $\beta^{Ijlmp}(g_{ab})g_{lm,p}=0$. Thus,
\begin{equation}\label{finalalphameven}
\alpha^{Ij}(g_{ij},g_{ij,k},g_{ij,kl},D_I,D_{I,j})=\eta^{IjKl}(g_{ab})D_{K,l}
\end{equation}
as desired.
If $m$ is odd, we instead obtain $\beta^{Ijlmp}(g_{ab})g_{lm,p}+\eta^{IjK}(g_{ab})D_K=0$, so that
\begin{equation}\label{alphamodd}
\alpha^{Ij}(g_{ij},g_{ij,k},g_{ij,kl},D_I,D_{I,j})=\beta^{Ij}(g_{ab})+\beta^{Ijlmpq}(g_{ab})g_{lm,pq}+\eta^{IjKl}(g_{ab})D_{K,l}.
\end{equation}
Taking $D=0$ in \eqref{alphamodd}, we see that $\beta^{Ij}(g_{ab})+\beta^{Ijlmpq}(g_{ab})g_{lm,pq}=\alpha^{Ij}(g_{ij},g_{ij,k},g_{ij,kl},0,0)$ form the components of an $(m+1,0)$-tensor, and, as a consequence, the terms $\beta^{Ijlmpq}(g_{ab})$ obey the symmetries \eqref{gammasyms1} in the last four indices. Thus, by the same arguments we used to show that $\gamma^{ijklqrst}(g_{ab})=0$, we conclude that $\beta^{Ijlmpq}(g_{ab})=0$, and, consequently,
\begin{equation}\label{alphadesired}
\alpha^{Ij}(g_{ij},g_{ij,k},g_{ij,kl},D_I,D_{I,j})=\beta^{Ij}(g_{ab})+\eta^{IjKl}(g_{ab})D_{K,l},
\end{equation}
as desired.

\hspace{6mm} Now, from the tensoriality of $\beta^{Ij}(g_{ab})$ and $\alpha^{Ij}$, it clearly follows that the terms $\eta^{IjKl}(g_{ab})D_{K,l}$ transform tensorially for all $(M,g,D)$, and it again follows from arguments identical to those in Lemma \ref{weyllemma} that $\eta^{IjKl}(g_{ab})=\frac{\partial\alpha^{Ij}}{\partial y_{Kl}}(g_{ab})$ give the components of a type $(2m+2,0)$-tensor field as well. 

\hspace{6mm} Given a triple $(M^n,g,D),$ a point $p \in M,$ and a coordinate system $(x^i)$ about $p$ for which $x^i(0)=0,$ we again observe that, on some neighborhood of $p$, $\exists$ a coordinate system $(\tilde{x}^i)$ with $\tilde{x}^i(p)=0$ and $x^i=\tilde{x}^i+\frac{1}{2}b^i_{jk}\tilde{x}^j\tilde{x}^k$ (where $(b^i_{jk})\in (\mathbb{R}^n)^{\otimes 3}$ satisfying $b^i_{jk}=b^i_{kj}$ is arbitrary). Under this change of coordinates, we evidently have $\frac{\partial x^i}{\partial \tilde{x}^j}(p)=\delta^i_j$ and $\frac{\partial^2 x^i}{\partial \tilde{x}^j\partial\tilde{x}^k}(p)=b^i_{jk},$ so $\tilde{g}_{ab}(p)=g_{ab}(p),$ 
$$\tilde{D}_{K,l}(p)=D_{K,l}(p)+(b^r_{k_1l}D_{rk_2\cdots k_m}(p)+\cdots+b^r_{k_ml}D_{k_1\cdots k_{m-1}r}(p)),$$
and, by the tensoriality of $\eta^{IjKl}D_{K,l},$ 
\begin{eqnarray*}
\eta^{IjKl}(g_{ab}(p))D_{K,l}(p)&=&\eta^{IjKl}(\tilde{g}_{ab}(p))\tilde{D}_{K,l}(p)\\
&=&\eta^{IjKl}(g_{ab}(p))\left(D_{K,l}(p)+b^r_{k_1l}D_{rk_2\cdots k_m}(p)+\cdots+b^r_{k_ml}D_{k_1\cdots k_{m-1}r}(p)\right).
\end{eqnarray*}
We conclude that, for any choice of $(M^n,g,D)$ and $(b^i_{jk})$ of the given form, in all coordinate charts on $M$, we have
\begin{equation}\label{antisymmcond}
\eta^{IjKl}(g_{ab})(b^r_{k_1l}D_{rk_2\cdots k_m}+\cdots+b^r_{k_ml}D_{k_1\cdots k_{m-1}r})=0.
\end{equation} 

\hspace{6mm} Fix an arbitrary point $p$ in a Lorentz manifold $(M^n,g,D),$ and fix a coordinate system about $p$. Given any $T\in (\mathbb{R}^n)^{\otimes m},$ note that we can choose $D \in \mathscr{T}^0_m(M)$ such that $D_K(p)=T_K$ in the given coordinate system; hence, we can replace $\eqref{antisymmcond}$ with
\begin{equation}\label{antisymmcond'}
\eta^{IjKl}(g_{ab})(b^r_{k_1l}T_{rk_2\cdots k_m}+\cdots +b^r_{k_ml}T_{k_1\cdots k_{m-1}r}),
\end{equation}
where $T \in (\mathbb{R}^n)^{\otimes m}$ is constant. We'll use this to show that $\eta^{IjKl}$ is antisymmetric in the last $m+1$ indices.

\hspace{6mm} Fix some multi-index $Q=q_1\cdots q_m$ ($0 \leq q_i \leq n-1$), and define an element $T \in (\mathbb{R}^n)^{\otimes m}$ by setting $T_Q=1$ and $T_K=0$ for $K\neq Q$. In this case, \eqref{antisymmcond'} gives
\begin{equation}\label{antisymmcond''}
\eta^{Ijk_1q_2\cdots q_ml}(g_{ab})b_{k_1l}^{q_1}+\cdots +\eta^{Ijq_1\cdots q_{m-1}k_ml}(g_{ab})b_{k_ml}^{q_m}=0.
\end{equation}
Fixing $t \in \{0,\ldots,n-1\}$ and setting $b^t_{tt}=1,$ $b^i_{jk}=0$ otherwise in the relation above, we obtain
\begin{equation}\label{bacond1}
\Sigma_{r=1}^m\eta^{Ijq_1\cdots q_{r-1}tq_{r+1}\cdots q_mt}(g_{ab})\delta_{q_rt}=deg_t(Q)\eta^{IjQt}(g_{ab})=0,
\end{equation}
from which it follows that $\eta^{IjQt}=0$ whenever $t$ occurs in $Q$. 

Next, fix some distinct $t,s \in \{0,\ldots,n-1\}$, and set $b^t_{ts}=b^t_{st}=1,$ $b^i_{jk}=0$ otherwise, so that \eqref{antisymmcond''} gives
\begin{equation}\label{bacond2}
\Sigma_{r=1}^m(\eta^{IjQs}(g_{ab})+\eta^{Ijq_1\cdots q_{r-1}sq_{r+1}\cdots q_mt}(g_{ab}))\delta_{q_rt}=0.
\end{equation}
Now, if $deg_t(Q)>1,$ then $t$ occurs in $q_1\cdots q_{r-1}sq_{r+1}\cdots q_m$ for every $r$, so by \eqref{bacond1}, $\eta^{Ijq_1\cdots q_{r-1}sq_{r+1}\cdots q_mt}=0,$ and it follows from \eqref{bacond2} that $\eta^{IjQs}=0.$ If $deg_t(Q)=1,$ with $q_r=t,$ then we simply obtain $\eta^{IjQs}=-\eta^{Ijq_1\cdots q_{r-1}sq_{r+1}\cdots q_mq_r}.$ Putting all this together, we've now shown that, for any multi-index $IjKl$ and $1\leq r \leq m,$ 
$$\eta^{Ijk_1\cdots k_{r-1}lk_{r+1}\cdots k_mk_r}=-\eta^{IjKl},$$
from which it follows that $\eta^{IjKl}$ is fully antisymmetric in its last $m+1$ indices (and, consequently--since $\eta^{IjKl}:=\frac{\partial^2L}{\partial y_{I,j}\partial y_{K,l}}$--in its first $m+1$ indices as well).

Now we wish to show that $\eta^{IjKl}(g_{ab})D_{I,j}D_{K,l}=c|d\gamma|_g^2,$ where $|\cdot|_g$ is the usual norm on $(m+1)$-tensors and $\gamma=Alt(D)$ is the fully antisymmetric part of $D$. By the antisymmetries of the terms $\eta^{IjKl},$ clearly we can assume $m+1 \leq n,$ and we see that 
$$\eta^{IjKl}(g_{ab})D_{I,j}D_{K,l}=c_m\Sigma_{i_1<\cdots<i_m<j}\Sigma_{k_1<\cdots<k_m<l}\eta^{IjKl}(g_{ab})(d\gamma)_{Ij}(d\gamma)_{Kl}$$
for some constant $c_m$ depending only on our convention for the definition of $Alt$, since $(d\gamma)_{i_1\cdots i_{m+1}}=c_m'\Sigma_{\sigma \in S_{m+1}}sgn(\sigma)\partial_{i_{\sigma(1)}}D_{i_{\sigma(2)}\cdots i_{\sigma(m)}}$. By the tensoriality of $\eta^{IjKl}(g_{ab})$, we know that $\eta^{IjKl}(\eta_{ab})$ is invariant under the action of the Lorentz group on $(2(m+1),0)$-tensors. Thus, by the same Lorentz-invariance arguments we used in the proof of Theorem \ref{thm01}, we have $\eta^{IjKl}(\eta_{ab})=0$ if $deg_t(IjKl)$ is odd for any $t$, so, combining this with the antisymmetries of $\eta^{IjKl}$, we see that
\begin{equation}\label{etalorentzinvar}
\eta^{IjKl}(g_{ab}(p))D_{I,j}(p)D_{K,l}(p)=c_m\Sigma_{i_1<\cdots<i_m<j}\eta^{IjIj}(\eta_{ab})(d\gamma)^2_{Ij}(p)
\end{equation}
for a coordinate system about $p$ satisfying $g_{ab}(p)=\eta_{ab}$.
By the Lorentz invariance and antisymmetries of $\eta^{IjIj}(\eta_{ab})$, it's also easy to see, as before, that $\eta^{IjIj}(\eta_{ab})=\eta^{I'j'I'j'}(\eta_{ab})$ when $I'j'=\sigma(i_1)\cdots \sigma(i_m)\sigma(j)$ for some permutation $\sigma$ of $\{0,\ldots,n-1\}$ fixing $0$, and $\eta^{I0I0}=(-1)^{deg_{\sigma(0)}I+1}\eta^{I'j'I'j'}$ when $I'j'=\sigma(i_1)\cdots \sigma(i_m)\sigma(0)$ for a permutation $\sigma$ that doesn't fix $0$. Hence, we indeed have
$$\eta^{IjKl}(g_{ab})D_{I,j}D_{K,l}=-c_m\eta^{012\cdots (m+1)012\cdots (m+1)}(\eta_{ab})\Sigma_{i_1<\cdots<i_m<j}\eta_{i_1i_1}(d\gamma)_{Ij}^2=c|d\gamma|_g^2,$$
as desired.

\hspace{6mm} We've now shown that
$$f_{g,D}=a+bR_g+c|d\gamma|_g^2+\beta^{Ij}(g_{ab})D_{I,j}+\zeta^I(g_{ij},g_{ij,k},g_{ij,kl},D_J)D_I,$$
where $\zeta^I=\frac{\partial L}{\partial y_I}.$ Since $f_{g,D}-(a+bR_g+c|d\gamma|^2)$ is another invariant function, it follows again from the same arguments we used in Lemma \ref{weyllemma} that $\beta^{Ij}(g_{ab})$ is tensorial, so $Bd_g(D)=:\beta^{Ij}(g_{ab})D_{I;j}$ gives an invariant divergence term. By considering normal coordinate systems, we conclude that
$$f_{g,D}=a+bR_g+c|d\gamma|_g^2+Bd_g(D)+\mu^I(g_{ij},g_{ij,kl},D_J)D_I,$$
where $\mu^I=\zeta^I(\cdot, 0,\cdot,\cdot)$. 

\hspace{6mm} It follows that $\mu^I(g_{ij},g_{ij,kl},D_J)$ defines a $(m,0)$-tensor field, and by (\ref{algconstraint}), we know that $\mu^I$ has the form 
$$\mu^I(g_{ij},g_{ij,kl},D_J)=\lambda^I(g_{ab})+\lambda^{Ilmpq}(g_{ab})g_{lm,pq}+\lambda^{IJ}(g_{ab})D_J.$$
Setting $D=0$, we see that $\lambda^I(g_{ab})+\lambda^{Ilmpq}(g_{ab})g_{lm,pq}$ defines a tensor field in its own right, and we can apply the same arguments we used to show that $\frac{\partial^2\hat{L}}{\partial x_{ijkl}\partial x_{qrst}}=0$ to conclude that $\lambda^{Ilmpq}=0.$ Finally, since $\lambda^I(g_{ab})$ and $\lambda^{IJ}(g_{ab})D_J$ determine $(m,0)$-tensor fields, setting $T_g(D)=\lambda^I(g_{ab})D_I$ and $Q_g(D)=\lambda^{IJ}(g_{ab})D_ID_J,$ we arrive at the desired form:
$$f_{g,D}=a+bR_g+c|d\gamma|_g^2+Bd_g(D)+Q_g(D)+T_g(D).$$
\end{proof}

Since the divergence term $Bd_g(D)$ has no effect on the variational principle arising from the Lagrangians in question, it follows that, for all such principles, the $m$-form $\gamma$ is the only part of $D$ whose dynamics are controlled by the Euler-Lagrange equations--a curiosity worth examining from a physical perspective. We suspect that similar results will hold under a variety of slightly weaker algebraic restrictions on the Lagrangian $L$.

\end{document}